\documentclass[reqno,11pt,a4paper,english]{amsart}

\usepackage{amsmath} % utilities for mathematics
\usepackage{amsthm} % utilities for theorem environment
\usepackage{amssymb} % loads mathematical fonts
\usepackage[dvipsnames]{xcolor}
\usepackage{amscd} % utilities for commutative diagrams
\usepackage{mathtools}
\usepackage[notcite,notref]{}
\usepackage[ddmmyyyy,hhmmss]{datetime}
\usepackage[colorinlistoftodos,bordercolor=orange,backgroundcolor=orange!20,linecolor=orange,textsize=scriptsize]{todonotes}

\usepackage{subcaption}% <-- added

\usepackage{array} % core package
\usepackage{newtxtext}
\usepackage[varvw]{newtxmath}
\usepackage[cal=boondoxo,scr=euler]{mathalfa}
\usepackage[linktocpage]{hyperref} % permits navigating on the pdf
\usepackage{cleveref} % smart referencing
\usepackage{caption} %
\usepackage{graphics,graphicx} % permits putting images
\usepackage{tikz,tikz-cd} % tool for diagrams

\usetikzlibrary{fit,matrix,graphs,graphs.standard,calc,shapes.geometric, arrows}
\usetikzlibrary{decorations.pathreplacing,}%angles,quotes

\usetikzlibrary{automata}
\usepackage[shortlabels]{enumitem} % permits enumerating using letters or other symbols

\DeclareMathAlphabet{\mathsf}{OT1}{\sfdefault}{m}{n}

\newcommand{\nocontentsline}[3]{}
\newcommand{\tocless}[2]{\bgroup\let\addcontentsline=\nocontentsline#1{#2}\egroup}

\usepackage[margin=1.24in]{geometry}
\usepackage{verbatim}

\usepackage{scalerel}

\usepackage{multirow}

\makeatletter
\def\dual#1{\expandafter\dual@aux#1\@nil}
\def\dual@aux#1/#2\@nil{\begin{tabular}{@{}c@{}}#1\\#2\end{tabular}}
\makeatother

\makeatletter
\@namedef{subjclassname@2020}{\textup{2020} Mathematics Subject Classification}
\makeatother

\DeclareMathAlphabet{\amathbb}{U}{bbold}{m}{n}

\hypersetup{
    colorlinks = true,
    linkbordercolor = {white},
    linkcolor = {BrickRed},
    anchorcolor = {black},
    citecolor = {BrickRed},
    filecolor = {cyan},
    menucolor = {BrickRed},
    runcolor = {cyan},
    urlcolor = {black}
}

\tikzstyle{rectan} = [rectangle, rounded corners, 
minimum width=1.5cm, 
minimum height=0.75cm,
text width=3cm,
text centered, 
draw=black,
font = \footnotesize,
]

\tikzstyle{ghost} = [circle, 
minimum width=1pt, 
minimum height=1pt,
text width=1pt,
text centered, 
draw=black,
font = \footnotesize,
]

\newtheoremstyle{teoremas}% <name>
{8pt}% <Space above>
{8pt}% <Space below>
{\itshape}% <Body font>
{}% <Indent amount>
{\bfseries}% <Theorem head font>
{}% <Punctuation after theorem head>
{.5em}% <Space after theorem headi>
{}% <Theorem head spec (can be left empty, meaning `normal')>

\theoremstyle{teoremas}
\newtheorem{theorem}{Theorem}[section]
\newtheorem{corollary}[theorem]{Corollary}
\newtheorem{lemma}[theorem]{Lemma}
\newtheorem{proposition}[theorem]{Proposition}

\newtheoremstyle{definition}% <name>
{8pt}% <Space above>
{8pt}% <Space below>
{}% <Body font>
{}% <Indent amount>
{\bfseries}% <Theorem head font>
{}% <Punctuation after theorem head>
{.5em}% <Space after theorem headi>
{}% <Theorem head spec (can be left empty, meaning `normal')>

\theoremstyle{definition}
\newtheorem{definition}[theorem]{Definition}
\newtheorem{conjecture}[theorem]{Conjecture}

\newtheorem{question}[theorem]{Question}
\newtheorem{example}[theorem]{Example}
\newtheorem{remark}[theorem]{Remark}

\crefname{theorem}{theorem}{theorems}
\Crefname{theorem}{Theorem}{Theorems}
\crefname{lemma}{lemma}{lemmas}
\Crefname{lemma}{Lemma}{Lemmas}
\crefname{proposition}{proposition}{propositions}
\Crefname{proposition}{Proposition}{Propositions}

\DeclareMathOperator{\rk}{rk}

\newcommand{\M}{\mathsf{M}}

\newcommand{\Q}{\mathbb{Q}}
\newcommand{\trunc}{\operatorname{trunc}}

\newcommand{\Z}{\mathbb{Z}}

\newcommand{\Hilb}{\operatorname{Hilb}}

\renewcommand{\H}{\mathrm{H}}
\newcommand{\CH}{\mathrm{CH}}
\newcommand{\aug}{\operatorname{aug}}
\newcommand{\Int}{\operatorname{Int}}
\newcommand{\IH}{\mathrm{IH}}
\newcommand{\uH}{\underline{\mathrm{H}}}
\newcommand{\uCH}{\underline{\mathrm{CH}}}

\newcommand{\asc}{\operatorname{asc}}

\newcommand{\des}{\operatorname{des}}
\newcommand{\rev}{\operatorname{rev}}
\renewcommand{\aug}{\operatorname{aug}}

\AtBeginDocument{%
   \def\MR#1{}
}

\title[Chow functions for partially ordered sets]{Chow functions for partially ordered sets}

\author[L. Ferroni]{Luis Ferroni}

\address{(L. Ferroni)
   Dipartimento di Matematica, Universit\`a di Pisa, Pisa, Italy.
}
\email{luis.ferroni@unipi.it}

\thanks{Luis Ferroni was a member at the Institute for Advanced Study, funded by the Minerva Research Foundation. Jacob Matherne received support from NSF Grant DMS-2452179 and Simons Foundation Travel Support for Mathematicians Award MPS-TSM-00007970.}

\author[J.~P.~Matherne]{Jacob P. Matherne}

\address{(J. P. Matherne)
Department of Mathematics, North Carolina State University, Raleigh, NC, USA.
}
\email{jpmather@ncsu.edu}

\author[L. Vecchi]{Lorenzo Vecchi}
\address{(L. Vecchi)
  Department of Mathematics, KTH Royal Institute of Technology, Stockholm, Sweden.
}

\email{lvecchi@kth.se}


\allowdisplaybreaks

\begin{document}

\begin{abstract}
     Three decades ago, Stanley and Brenti initiated the study of the Kazhdan--Lusztig--Stanley (KLS) functions, putting on common ground several polynomials appearing in algebraic combinatorics, discrete geometry, and representation theory. In the present paper we develop a theory that parallels the KLS theory. To each kernel in a given poset, we associate a polynomial function that we call the \emph{Chow function}. The Chow function often exhibits remarkable properties, and sometimes encodes the graded dimensions of a cohomology or Chow ring.  The framework of Chow functions provides natural polynomial analogs of graded module decompositions that appear in algebraic geometry, but that work for arbitrary posets, even when no graded module decomposition is known to exist.  In this general framework, we prove a number of unimodality and positivity results without relying on versions of the Hard Lefschetz theorem.  Our framework shows that there is an unexpected relation between positivity and real-rootedness conjectures about chains on face lattices of polytopes by Brenti and Welker, Hilbert--Poincar\'e series of matroid Chow rings by Ferroni and Schr\"oter, and flag enumerations on Bruhat intervals of Coxeter groups by Billera and Brenti.
\end{abstract}

%\date{\today~at \currenttime}

\maketitle

\setcounter{tocdepth}{1}
\tableofcontents %This is only to facilitate navigating (in my opinion the article is not long enough to deserve a table of contents on its first page)

\section{Introduction}\label{sec:one}

\subsection{Overview}
In the foundational paper \cite{stanley-local}, Stanley developed a notable framework to study polynomials arising from partially ordered sets. This puts on common ground and unifies several---a priori unrelated---theories that are of fundamental importance in mathematics. Three prominent examples are i) the enumeration of points, lines, planes, etc. in a matroid, ii) the enumeration of faces in convex polytopes, and iii) the combinatorics and representation theory associated to Coxeter groups.

Following another influential paper by Brenti~\cite{brenti-kls}, we will refer to this as the Kazhdan--Lusztig--Stanley (KLS) theory for posets. We point to a recent survey by Proudfoot \cite{proudfoot-kls} for a self-contained introduction to KLS theory and its algebro-geometric consequences. In what follows we summarize the basic setup of the KLS theory, following closely the notation of \cite{proudfoot-kls}. In Section~\ref{sec:two} we provide more details about the construction of the main objects.

Assume that $P$ is a locally finite, weakly ranked, partially ordered set, and let $\Int(P)$ be the set of all closed intervals of $P$. We denote by $\rho\colon \Int(P)\to \mathbb{Z}$ the weak rank function of $P$. Consider the incidence algebra $\mathcal{I}(P)$ of $P$ over the univariate polynomial ring $\mathbb{Z}[x]$. The weak rank function $\rho$ gives rise to the subalgebra $\mathcal{I}_{\rho}(P)\subseteq \mathcal{I}(P)$ consisting of the elements $f\in \mathcal{I}(P)$ such that $\deg f_{st} \leq \rho_{st}$ for each closed interval $[s,t]$.   Stanley realized the importance of special elements $\kappa\in \mathcal{I}_{\rho}(P)$ which are called $(P,\rho)$-kernels or, when $\rho$ is understood from context, just $P$-kernels. To each such kernel $\kappa$ one associates two important elements $f, g \in \mathcal{I}_{\rho}(P)$.  The element $f$ (resp. $g$) is often called the right (resp. left) Kazhdan--Lusztig--Stanley (KLS) function associated to the $(P,\rho)$-kernel $\kappa$.

In each of the three examples mentioned in the first paragraph, the posets and the kernels are, respectively,  i) the lattice of flats of a matroid with the characteristic function as kernel, ii) the face lattice of a convex polytope with the kernel $[s,t]\mapsto (x-1)^{\dim t - \dim s}$, and iii) the strong Bruhat order poset of a Coxeter group with the $R$-polynomials as kernel. In these three cases the posets are graded and bounded, and the assignment $[s,t]\mapsto \rho_{st}$ is given by the length of an arbitrary saturated chain starting at $s$ and ending at $t$. Correspondingly, the KLS functions that arise in each of these cases are i) the Kazhdan--Lusztig polynomial of the matroid defined by Elias, Proudfoot, and Wakefield in \cite{elias-proudfoot-wakefield}, ii) the toric $g$-polynomial of the polytope introduced by Stanley~\cite{stanley-toric-g}, and iii) the Kazhdan--Lusztig polynomial(s) of the Coxeter group discovered by Kazhdan and Lusztig \cite{kazhdan-lusztig}.  

The central contribution of the present work is the introduction of a new class of functions, that we call \emph{Chow functions}, associated to any $(P,\rho)$-kernel $\kappa$. As opposed to the case of the KLS functions, where a convention of left versus right constitutes an essential part of the definition, in our case there is a single distinguished element $\H\in \mathcal{I}_{\rho}(P)$ called the $\kappa$-Chow function associated to $(P,\rho)$. Notably, the KLS functions are required to satisfy a very restrictive degree bound: $\deg f_{st} < \frac{1}{2} \rho_{st}$ and $\deg g_{st} < \frac{1}{2}\rho_{st}$ for each $s < t$. In our case, the Chow function $\H$ satisfies a weaker degree bound: $\deg \H_{st} < \rho_{st}$ for each $s<t$; in order to compensate for the additional degrees of freedom, one imposes that the polynomials $\H_{st}$ are \emph{palindromic}. 

Chow functions and KLS functions turn out to be tightly connected to each other. Often, properties of one have an impact on the other. The most significant example of this phenomenon in the present paper is the following result.  

\begin{theorem}\label{thm:kls-positive-chow-unimodal-intro}
    Let $\kappa$ be a $(P,\rho)$-kernel. If the right KLS function $f$ or the left KLS function $g$ is non-negative, then the Chow function $\H$ is non-negative and unimodal.
\end{theorem}

Whenever we say that an element $a\in \mathcal{I}(P)$ is non-negative (resp. unimodal, symmetric, $\gamma$-positive, etc.) we mean that each of the polynomials $a_{st}(x)$ is non-negative (resp. unimodal, symmetric, $\gamma$-positive, etc.)

We prove Theorem~\ref{thm:kls-positive-chow-unimodal-intro} motivated by a module decomposition called the \emph{canonical decomposition} of the matroid Chow ring in \cite{braden-huh-matherne-proudfoot-wang} (see Theorem~\ref{thm:ncd} below). Furthermore, our proof is entirely combinatorial, in the sense that we do not deal with any algebraic structures but only with polynomials. Notice that the above theorem yields unimodality results in the three aforementioned main examples, since the KLS functions were proven to be non-negative in the following substantial works: i) by Braden, Huh, Matherne, Proudfoot, and Wang \cite{braden-huh-matherne-proudfoot-wang} via the introduction of the matroid intersection cohomology\footnote{We note, however, that the \emph{left} KLS function is trivial in this case.}; ii) by Karu in \cite{karu04} building upon earlier work of McMullen \cite{mcmullen}, Barthel, Brasselet, Fieseler, and Kaup \cite{barthel-brasselet-fieseler-kaup}, and Bressler and Lunts \cite{bressler-lunts}; and iii) by Elias and Williamson \cite{elias-williamson} via the machinery of Soergel bimodules \cite{soergel}, and relying on techniques by De Cataldo and Migliorini \cite{decataldo-migliorini1,decataldo-migliorini2}.

The main inspiration behind the definition of Chow functions, and in fact the reason behind the choice of this name, stems from the first on-going example concerning matroids. The Chow function encodes the Hilbert series of the Chow rings of all minors of a matroid. These Chow rings were introduced by Feichtner and Yuzvinsky in \cite{feichtner-yuzvinsky} and played a primary role in the resolution of long-standing conjectures in combinatorics \cite{adiprasito-huh-katz,ardila-denham-huh,braden-huh-matherne-proudfoot-wang}; for amenable surveys we refer to \cite{okounkov-icm22,huh-icm22,ardila-icm22,eur}. The case of Chow functions arising from matroids was the main theme of a previous paper written in collaboration with Matthew Stevens \cite{ferroni-matherne-stevens-vecchi}. 

A further motivation to develop the theory in the present paper was to understand to what extent one can hope to derive other versions of some crucial module decompositions concerning matroid intersection cohomologies, by Braden, Huh, Matherne, Proudfoot, and Wang in \cite{semismall,braden-huh-matherne-proudfoot-wang}. We came to realize that a number of the module decompositions that constitute the intricate induction appearing in \cite{braden-huh-matherne-proudfoot-wang} can be shadowed step by step, but working instead with \emph{polynomials} rather than \emph{graded modules}. There are some advantages in this approach. 
\begin{itemize}
    \item Our framework does not require us to work with matroids nor posets with characteristic polynomials displaying any specific sign pattern in their coefficients, see Section~\ref{sec:four}. In particular, we can apply these constructions to the examples of face lattices of polytopes (Section~\ref{sec:five}) or Bruhat intervals (Section~\ref{sec:six}).
    \item We are able to state results that would not be possible to obtain by taking graded dimensions of any module or ring (see for example Theorem~\ref{thm:ncd}, Theorem~\ref{thm:aug-ncd}, and Theorem~\ref{thm:chow-from-kl}). A priori, our identities may involve polynomials that cannot possibly be Hilbert series or Poincar\'e polynomials, e.g., when one of the coefficients is negative.
    \item We are able to provide combinatorial proofs of statements that were known to be valid via the application of difficult results from algebraic geometry, and we achieve so for more general classes of posets (see the discussion around Theorem~\ref{thm:chi-chow-unimodal} and Theorem~\ref{thm:cohen-macaulay-chow-gamma-positive}).
    \item This framework is amenable for building upon intuition from one setting (say, polytopes) and using it in another (say, Coxeter groups). For example, the use of the $\mathbf{cd}$-index in the case of polytopes in Section~\ref{sec:five} led us to consider in Section~\ref{sec:six} the complete $\mathbf{cd}$-index of Bruhat intervals introduced by Billera--Brenti \cite{billera-brenti}.
\end{itemize}

In addition to the key object $\H\in \mathcal{I}_{\rho}(P)$ introduced in this paper, we also study two related functions: $F\in \mathcal{I}_{\rho}(P)$ the \emph{right augmented Chow function}, and $G\in \mathcal{I}_{\rho}(P)$ the \emph{left augmented Chow function}. These are obtained by convolving $\H$ with the right and left KLS functions respectively (see Section~\ref{sec:augmented} for the details), and they also exhibit remarkable properties.  The element $G$ plays a key role in the singular Hodge theory of matroids \cite{braden-huh-matherne-proudfoot-wang}, where it encodes the Hilbert--Poincar\'e series of augmented Chow rings (hence the name), but $F$ is more subtle. Nonetheless, in the general context of this paper, there is no reason to prefer $G$ over $F$, so we develop the theory in full generality. It is natural to formulate questions concerning what algebro-geometric objects they model, and we do so in Section~\ref{subsec:chowish-ring}.

\subsection{The three main examples}

A priori, we do not require the poset $P$ to be bounded. However, in some important cases we find ourselves in this situation and will correspondingly denote $\widehat{0} = \min P$ and $\widehat{1} = \max P$. We will refer to $\H_{\widehat{0}\widehat{1}}(x)$ as the \emph{Chow polynomial} of $P$. 

An important feature of Chow functions is that they are symmetric. More precisely, each of the polynomials $\H_{st}(x)$ satisfies the identity
    \[ \H_{st}(x) = x^{\rho_{st} - 1} \H_{st}(x^{-1}), \qquad \text{ for all $s < t$ in $P$}.\]
Without imposing additional restrictions on the poset or the kernel, Chow functions may fail to be unimodal, and in fact the coefficients of $\H_{st}(x)$ can even be negative, but Theorem~\ref{thm:kls-positive-chow-unimodal-intro} gives a striking criterion for unimodality.

\subsubsection{Characteristic Chow functions}

As is pointed out in \cite{proudfoot-kls}, the characteristic function $\chi\in \mathcal{I}_{\rho}(P)$ is a $P$-kernel in any weakly ranked locally finite poset $P$, and lattices of flats of matroids are just a special case. In particular, there is no formal obstruction in considering the KLS functions $f$ and $g$ and the Chow function $\H$ arising from this setup. For the sake of clarity, we will refer to this Chow function $\H$ as the \emph{characteristic Chow function} or, for brevity, the \emph{$\chi$-Chow function} of $(P,\rho)$. In the matroid setting, one has the following result.

\begin{theorem}
    Let $\M$ be a loopless matroid and let $P=\mathcal{L}(\M)$ be its lattice of flats. Then the characteristic Chow polynomial of $P$ coincides with the Hilbert--Poincar\'e series of the Chow ring of $\M$. In particular, it is unimodal.
\end{theorem}

The first part of the above statement is proved in \cite{ferroni-matherne-stevens-vecchi}, whereas the second follows from the validity of the Hard Lefschetz theorem, proved by Adiprasito, Huh, and Katz \cite{adiprasito-huh-katz}. 

In \cite{ferroni-matherne-stevens-vecchi} a strengthening of unimodality is also proved: the Hilbert--Poincar\'e series of a matroid Chow ring is in fact $\gamma$-positive \cite[Theorem~1.8]{ferroni-matherne-stevens-vecchi}. The main tool to prove that was a key result of Braden, Huh, Matherne, Proudfoot, and Wang \cite{semismall}, who established a semi-small decomposition for the Chow ring of a matroid. 

In the present paper we deal with much more general posets, for which the Chow ring is not even defined. By applying our numerical analog of the canonical decomposition of matroid Chow rings from \cite{braden-huh-matherne-proudfoot-wang} we have the following result.

\begin{theorem}
    Let $P$ be any graded bounded poset. The $\chi$-Chow polynomial of $P$ is unimodal.
\end{theorem}

Notice that this can be viewed as a corollary of Theorem~\ref{thm:kls-positive-chow-unimodal-intro}, because the left KLS function is identically $1$. The latter fact is just equivalent to the inclusion-exclusion principle. Most of the previous proofs of the above unimodality result (for geometric lattices only) relied on versions of the Hard Lefschetz theorem.

Besides unimodality, one may consider the stronger property of being $\gamma$-positive. For geometric lattices this property is known to hold via \cite[Theorem~1.8]{ferroni-matherne-stevens-vecchi}, but in general it does not (see Example~\ref{example:non-gamma-positive} below). We go far beyond geometric lattices and prove the following.  

\begin{theorem}
    Let $P$ be any Cohen--Macaulay poset. The $\chi$-Chow polynomial of $P$ is $\gamma$-positive.
\end{theorem}

For a general Cohen--Macaulay poset there is no obvious way of defining the Chow ring and therefore no clear analogue of the semi-small decomposition of \cite{semismall}. In fact, we prove a statement more general than the above result. We show that the $\gamma$-polynomial associated to the $\chi$-Chow function is a non-negative specialization of the flag $h$-vector (see Theorem~\ref{thm:chow-from-flag-h}). This also generalizes a beautiful result by Stump \cite[Theorem~1.1]{stump}, which was a key motivation for our proof. Ferroni and Schr\"oter \cite{ferroni-schroter} and, independently, Huh and Stevens \cite{stevens-bachelor} conjectured that whenever $P=\mathcal{L}(\M)$ is the lattice of flats of a matroid $\M$, then the Hilbert--Poincar\'e series of the Chow ring of $\M$ is a real-rooted polynomial. 

\begin{conjecture}\label{conj:char-chow-real-rooted}
    Let $P$ be any geometric lattice. The $\chi$-Chow polynomial of $P$ is real-rooted.
\end{conjecture}

In a previous version of this manuscript we postulated as a conjecture that larger classes of posets, such as Cohen--Macaulay posets, could also have real-rooted $\chi$-Chow polynomials. This turns out to be false (see Remark~\ref{rem:not-RR-CM}).

\subsubsection{Eulerian Chow functions}

Whenever $P$ is an Eulerian poset, the element $\varepsilon\in \mathcal{I}_{\rho}(P)$ given by $\varepsilon_{st} = (x-1)^{\rho_{st}}$ is a $P$-kernel. The resulting Chow function will be called the \emph{Eulerian Chow function}, or \emph{$\varepsilon$-Chow function} for brevity, associated to $P$. We prove the following result.

\begin{theorem}
    The Eulerian Chow polynomial of $P$ equals the $h$-polynomial of the barycentric subdivision of $P$.
\end{theorem}

By barycentric subdivision of a poset $P$, we mean the simplicial complex whose faces are the flags of elements of $P$. We do not know whether Eulerian Chow polynomials are always non-negative. Moreover, we explain why we expect this question to be very subtle. By the positivity of the KLS functions proved in certain special cases (e.g., for face posets of simplicial polytopes \cite{stanley-g-theorem}, of general polytopes \cite{karu04}, or of simplicial spheres \cite{adiprasito,papadakis-petrotou}), Theorem~\ref{thm:kls-positive-chow-unimodal-intro} guarantees that the $\varepsilon$-Chow function is non-negative and unimodal. However, another deep result by Karu \cite{karu} about the $\mathbf{cd}$-index of Gorenstein* posets (that is, posets that are both Eulerian and Cohen--Macaulay) can be used to obtain the following stronger property.

\begin{theorem}
    Let $P$ be a Gorenstein* poset. The $\varepsilon$-Chow function of $P$ is $\gamma$-positive.
\end{theorem}

It is natural to ask whether the above property can be upgraded to real-rootedness. That is equivalent to a long-standing folklore conjecture, posed as an open question by Brenti and Welker \cite{brenti-welker}, when $P$ is the face poset of a polytope.

\begin{conjecture}[{see \cite[Question~1]{brenti-welker}}]\label{conj:gorenstein-real-rooted}
    Let $P$ be the face poset of a polytope (or even just a Gorenstein* poset). Then the $\varepsilon$-Chow polynomial of $P$ is real-rooted.
\end{conjecture}

The question for Gorenstein* posets is strongly related to questions formulated by Athanasiadis and Tzanaki \cite[Question~7.4]{athanasiadis-tzanaki} and by Athanasiadis and Kalampogia-Evangelinou \cite[Question~5.2]{athanasiadis-kalampogia}.

\subsubsection{Coxeter Chow functions}

The chief example of KLS functions are precisely the Kazhdan--Lusztig polynomials of Bruhat intervals, defined by Kazhdan and Lusztig in \cite{kazhdan-lusztig}. The kernels in this case are the so-called $R$-polynomials. A powerful result by Dyer \cite{dyer} allows for the computation of the $R$-polynomials via a computation on Bruhat graphs. We use this to prove the following interpretation for the Chow function.

\begin{theorem}\label{thm:comb-coxeter-H-intro}
    Let $W$ be a Coxeter group with a reflection order $<$ and two elements $u, v \in W$. Then 
    \[
    \H_{uv}(x) = \sum_{\Delta \in B(u,v)}x^{\frac{\rho_{uv}-\ell(\Delta)}{2} + \asc(\Delta)} = \sum_{\Delta \in B(u,v)}x^{\frac{\rho_{uv}-\ell(\Delta)}{2} + \des(\Delta)}.
    \]
\end{theorem}

In the above statement $B(u,v)$ stands for all the paths in the Bruhat graph of $W$ that go from $u$ to $v$, $\ell(\Delta)$ stands for the length of the path $\Delta$, $\des$ stands for the number of descents of the path, whereas asc stands for the number of ascents. In particular, the Chow function  enumerates these paths according to a descent-like statistic. We show that the combinatorial invariance conjecture for Chow functions is equivalent to the combinatorial invariance conjecture for Kazhdan--Lusztig or $R$-polynomials, see Theorem~\ref{thm:comb-invariance}.

Thanks to the breakthrough of Elias and Williamson~\cite{elias-williamson}, and as a consequence of Theorem~\ref{thm:kls-positive-chow-unimodal-intro}, we obtain that the above enumeration of paths yields a unimodal polynomial. By shadowing the discussion of the two previous examples, we are led to consider $\gamma$-positivity and real-rootedness. In the case of polytopes (or Gorenstein* posets), the key tool to prove $\gamma$-positivity is the result on the $\mathbf{cd}$-index proved by Karu \cite{karu}. In this case, we need to rely on a more complicated non-commutative polynomial called the \emph{complete $\mathbf{cd}$-index}, introduced by Billera and Brenti \cite{billera-brenti}. We prove the following. 

\begin{theorem}
    Let $W$ be a Coxeter group and let $u < v$ in $W$. The $\gamma$-polynomial associated to the Coxeter Chow polynomial $\H_{uv}$ is a positive specialization of the complete $\mathbf{cd}$-index of the interval $[u,v]$.
\end{theorem}

The precise positive specialization is proved in Corollary~\ref{coro:complete-cd-chow-gamma}. Billera and Brenti conjecture the non-negativity of all the coefficients of the complete $\mathbf{cd}$-index for any interval in a Coxeter group, see \cite[Conjecture~6.1]{billera-brenti}. Some special cases of that conjecture are known to be true (see, e.g., \cite{karu-complete-cd}, \cite{fan-cd}), but it remains open in general.   The preceding theorem implies that if Billera and Brenti's conjecture is true, then the Coxeter Chow functions of a Coxeter group are $\gamma$-positive. That is, we have the following conjecture. 

\begin{conjecture}\label{conj:coxeter-gamma-positive}
    Coxeter Chow polynomials of Bruhat intervals of Coxeter groups are $\gamma$-positive.
\end{conjecture}

Emboldened by Conjecture~\ref{conj:char-chow-real-rooted} and Conjecture~\ref{conj:gorenstein-real-rooted},  and numerous experiments on Bruhat intervals of rank $\leq 7$, we also pose the following conjecture.

\begin{conjecture}\label{conj:coxeter-real-rooted}
    Coxeter Chow polynomials of Bruhat intervals of Coxeter groups are always real-rooted.
\end{conjecture}

\subsection{Paper outline}

In Section~\ref{sec:two} we briefly recapitulate the key notions about polynomial inequalities and poset properties that we will need. 

The chief contribution of this paper is the combinatorial framework of Chow functions developed in Section~\ref{sec:three}; we view this construction as a counterpart to Stanley's development of the theory of KLS functions in \cite{stanley-local}. This section introduces Chow functions and augmented Chow functions, and here we study their general connection with KLS functions and $Z$-functions.  In this section we prove various numerical analogues of graded module decompositions appearing in \cite{braden-huh-matherne-proudfoot-wang}. Sections~\ref{sec:four},~\ref{sec:five}, and~\ref{sec:six} then apply the general machinery developed in Section~\ref{sec:three} to our three central running examples, as we now describe.

In Section~\ref{sec:four}, we introduce the $\chi$-Chow function and study a number of combinatorial properties that it and its augmented counterparts satisfy; furthermore, we use the case of matroids to explain the algebro-geometric motivation for the main results in Section~\ref{sec:three}.

Section~\ref{sec:five} describes the Chow function arising from an Eulerian poset (or, for concreteness, the face poset of a polytope). We discuss how it relates to barycentric subdivisions, and we prove that the Chow function in this example also satisfies strong inequalities, by relying on a deep theorem by Karu \cite{karu}. 

Section~\ref{sec:six} addresses the case of Coxeter groups: we give a combinatorial description of the Chow function, and we relate it to the complete $\mathbf{cd}$-index of Billera and Brenti \cite{billera-brenti}.

\section{Preliminaries}\label{sec:two}

\subsection{Inequalities for polynomials}

Let $p(x) = a_0 + a_1x+\cdots +a_m x^m\in \mathbb{Z}[x]$ denote a polynomial having non-negative coefficients. The polynomial $p(x)$ is said to be \emph{unimodal} if there exists an index $j$ such that 
    \[ a_0 \leq \cdots \leq a_{j-1} \leq a_j \geq a_{j+1} \geq \cdots \geq a_m.\]

We say that $p(x)$ is \emph{symmetric} if there exists some index $d$ such that $a_i = a_{d-i}$ for each $i$ (where $a_i:=0$ if $i<0$). In this case, we say that $p(x)$ has center of symmetry $d/2$. Notice that the symmetry of the coefficients can be encoded via the equation $p(x) = x^d p(1/x)$. For thorough references about unimodality, we refer to \cite{stanley-unimodality,brenti-unimodality,branden}.

The following statement provides a useful characterization of polynomials that are symmetric and unimodal.

\begin{lemma}\label{lemma:char-symmetric-unimodality}
    Let $p(x)$ be a polynomial with non-negative coefficients. The following are equivalent.
    \begin{enumerate}[\normalfont(i)]
        \item $p(x)$ is unimodal and symmetric with center of symmetry $d/2$.
        \item There exist non-negative numbers $c_0, \ldots, c_{\lfloor d/2\rfloor}$ such that
            \[ p(x) = \sum_{i=0}^{\lfloor d/2\rfloor} c_i\, x^i (1+x+\cdots+x^{d-2i}).\]
    \end{enumerate}
\end{lemma}

The proof is straightforward, so we omit it. A further property that will be of relevance in the present paper is that of $\gamma$-positivity. We say that the polynomial $p(x)$ is \emph{$\gamma$-positive} if it is symmetric with center of symmetry $d/2$ and there exist non-negative integers $\gamma_0,\ldots,\gamma_{\lfloor d/2\rfloor}$ such that
    \[ p(x) = \sum_{i=0}^{\lfloor d/2\rfloor} \gamma_i\, x^i (1+x)^{d-2i}.\]
It is not hard to see that a $\gamma$-positive polynomial is unimodal. We refer to \cite{athanasiadis-gamma-positivity} for a thorough survey on $\gamma$-positivity.  The \emph{$\gamma$-polynomial} associated to $p$ is defined by $\gamma(p,x) := \sum_{i=0}^{\lfloor d/2\rfloor} \gamma_i x^i$.
It satisfies the following property:
\[ p(x) = (1+x)^{d} \,\gamma\left(p,\frac{x}{(1+x)^2}\right).\]

If the polynomial $p(x)$ is symmetric and has only negative real roots, then it is $\gamma$-positive; see \cite[Remark~3.1]{branden}. In other words, for symmetric polynomials with non-negative coefficients we have the following (strict) hierarchy of properties:

\[ \text{real-rootedness} \Longrightarrow \text{$\gamma$-positivity} \Longrightarrow \text{unimodality}.\]

\subsection{Essential notions about posets}\label{sec:posetnotions}

Throughout this paper we will use the letter $P$ to denote a partially ordered set, and  $\Int(P)$ to denote the set of all closed intervals of $P$. We say that $P$ is \emph{locally finite} if for every pair of elements $s\leq t$ in $P$, the closed interval $[s,t]= \{w\in P : s\leq w \leq t\}$ has finitely many elements. The \emph{incidence algebra} of $P$, denoted by $\mathcal{I}(P)$, is the free $\mathbb{Z}[x]$-module spanned by $\Int(P)$. In other words, an element $a\in \mathcal{I}(P)$ associates to each closed interval $[s,t]\in \Int(P)$ a polynomial $a_{st}(x)\in \mathbb{Z}[x]$. Depending on the context we shall write $a_{st}$ or $a_{st}(x)$ interchangeably. The product (also known as convolution) of two elements $a,b\in \mathcal{I}(P)$ is defined via

    \[ (ab)_{st}(x) = \sum_{s\leq w\leq t} a_{sw}(x)\, b_{wt}(x) \qquad \text{ for every $s\leq t$ in $P$}.\]

The algebra $\mathcal{I}(P)$ satisfies the following basic properties:

\begin{enumerate}[(i)]
    \item The product in $\mathcal{I}(P)$ is associative but not commutative.
    \item There is a multiplicative identity $\delta\in \mathcal{I}(P)$ defined by
        \[ \delta_{st} = \begin{cases} 1 & \text{ if $s = t$},\\ 0 & \text{ if $s < t$}.\end{cases}\]
\end{enumerate}

The following is a simple criterion to decide whether an element in the incidence algebra of $P$ admits an inverse.

\begin{proposition}\label{prop:inverses-incidence-algebra}
    The element $a\in \mathcal{I}(P)$ admits a two-sided inverse, denoted $a^{-1}\in \mathcal{I}(P)$, if and only if $a_{ss}(x) = \pm 1$ for every $s\in P$. 
\end{proposition}

If we consider the element $\upzeta\in \mathcal{I}(P)$ defined by
    \[ \upzeta_{st} = 1 \qquad \text{ for all $s\leq t$},\]
the preceding proposition guarantees that it is invertible, and we will denote its inverse by $\mu = \upzeta^{-1}$. The element $\mu\in \mathcal{I}(P)$ is known as the \emph{M\"obius function} of $P$. It can alternatively be defined via the following recursion:
    \[ \mu_{st} = \begin{cases}
        1 & \text{ if $s=t$},\\
        - \sum_{s\leq w < t} \mu_{sw} & \text{ if $s < t$}.
    \end{cases}\]

An element $a\in \mathcal{I}(P)$ can satisfy an additional property called combinatorial invariance. Precisely, if $a_{st}(x) = a_{s't'}(x)$ whenever $[s,t]$ and $[s',t']$ are isomorphic as posets, then we say that $a$ is \emph{combinatorially invariant}. Note that $\delta$, $\mu$, and $\upzeta$ are combinatorially invariant.

\subsection{Weak rank functions and characteristic functions}

A \emph{weak rank function} on $P$ is a map $\rho\colon \Int(P) \to \Z_{\geq 0}$ satisfying the following properties:
    \begin{enumerate}[(i)]
        \item If $s < t$, then $\rho_{st} > 0$.
        \item If $s\leq w\leq t$, then $\rho_{st} = \rho_{sw} + \rho_{wt}$.
    \end{enumerate}
Observe that the second condition guarantees that $\rho_{ss} = 0$ for every $s\in P$. By definition, a \emph{weakly ranked poset} is a pair $(P,\rho)$ consisting of a partially ordered set $P$ and a weak rank function $\rho$ on $P$. We note explicitly that it is \emph{not} required that $\rho$ be combinatorially invariant. If $P$ has a minimum element $\widehat{0}$, we will often write $\rho(w) := \rho_{\widehat{0},w}$ for any $w \in P$.

A weak rank function $\rho$ on a locally finite poset $P$ induces a special subalgebra $\mathcal{I}_{\rho}(P)\subseteq \mathcal{I}(P)$ defined by
    \begin{equation} \label{eq:def-of-I-rho-P}
    \mathcal{I}_{\rho}(P) = \left\{a\in \mathcal{I}(P): \deg a_{st}(x) \leq \rho_{st} \text{ for all $s\leq t$ in $P$}\right\}.
    \end{equation}
This subalgebra admits an involution $a \mapsto a^{\rev}$ defined via the identity
    \begin{equation}
        \left(a^{\rev}\right)_{st}(x) = x^{\rho_{st}}\, a_{st}(x^{-1}).
    \end{equation}
The name ``rev'' stems from the fact that this involution reverses (with respect to the weak rank function) the coefficients of the polynomials associated to each interval.\footnote{We warn the reader that in other sources this involution is denoted by $a\mapsto \overline{a}$; however, in the present work we will reserve that notation for a different operation.} It is immediate from the definition that this involution respects products, that is, $(ab)^{\rev} = a^{\rev}\cdot b^{\rev}$. Similarly, whenever $a\in \mathcal{I}_{\rho}(P)$ is invertible, our involution commutes with taking inverses $(a^{-1})^{\rev} = \left(a^{\rev}\right)^{-1}$.

A key object in the subalgebra $\mathcal{I}_{\rho}(P)$ is the \emph{characteristic function}, denoted by $\chi$. It is defined by
    \begin{equation} \label{eq:char-function}
    \chi = \mu\cdot \upzeta^{\rev} = \upzeta^{-1}\cdot \upzeta^{\rev}.
    \end{equation}
More explicitly, to each interval $[s,t]$ of $P$ we associate the polynomial
    \[ \chi_{st}(x) = \sum_{s\leq w\leq t} \mu_{sw}\,x^{\rho_{wt}}.\]
Whenever $P$ is bounded, the polynomial $\chi_P(x) := \chi_{\widehat{0}\,\widehat{1}}(x)$ will often be called the \emph{characteristic polynomial of $P$}.

\subsection{The basics of KLS theory}

From the basic properties of the involution $\rev$ described in the previous subsection, we have that the characteristic function enjoys an important property:

\begin{equation}\label{eq:char-function-is-p-kernel} \chi^{\rev} = \left(\upzeta^{-1}\cdot \upzeta^{\rev}\right)^{\rev} = \left(\upzeta^{\rev}\right)^{-1} \cdot \upzeta = \left(\upzeta^{\rev}\right)^{-1} \cdot \left(\upzeta^{-1}\right)^{-1} = \left(\upzeta^{-1}\cdot \upzeta^{\rev}\right)^{-1} = \chi^{-1}.
\end{equation}

In other words, inverting $\chi$ just reverses its coefficients. This motivates a key definition.

\begin{definition}\label{def:p-kernel}
    Let $(P,\rho)$ be a weakly ranked poset. An element $\kappa\in \mathcal{I}_{\rho}(P)$ is said to be a \emph{$(P,\rho)$-kernel} if $\kappa_{ss}(x) = 1$ for all $s\in P$ and
        \[ \kappa^{-1} = \kappa^{\rev}.\]
    We say that $\kappa$ is \emph{non-degenerate} if $\deg \kappa_{st} = \rho_{st}$ for every $s\leq t$ in $P$.
\end{definition}

The notion of non-degeneracy appears to be new, but will be useful in many of our statements below.

It follows from the preceding discussion that the characteristic function $\chi\in \mathcal{I}_{\rho}(P)$ is a non-degenerate $(P,\rho)$-kernel. Furthermore, notice that reasoning as in equation~\eqref{eq:char-function-is-p-kernel}, it follows that if $a\in \mathcal{I}_{\rho}(P)$ is an invertible element and $\kappa := a^{-1}\cdot a^{\rev}$, then $\kappa$ is a $(P,\rho)$-kernel. Stanley proved in \cite[Theorem~6.5]{stanley-local} that all $(P,\rho)$-kernels arise in this way.

\begin{theorem}\label{thm:kls-functions}
    Let $\kappa\in \mathcal{I}_{\rho}(P)$ be a $(P,\rho)$-kernel. There exists a unique element $f\in \mathcal{I}_{\rho}(P)$ satisfying the following properties:
    \begin{enumerate}[\normalfont(i)]
        \item \label{it:f-i}$f_{ss}(x) = 1$ for all $s\in P$.
        \item \label{it:f-ii} $\deg f_{st}(x) < \frac{1}{2} \rho_{st}$ for all $s < t$.
        \item \label{it:f-iii} $f^{\rev} = \kappa\cdot f$.
    \end{enumerate}
    Similarly, there exists a unique element $g\in \mathcal{I}_{\rho}(P)$ satisfying the following properties:
    \begin{enumerate}[\normalfont(i')]
        \item \label{it:g-i} $g_{ss}(x) = 1$ for all $s\in P$.
        \item \label{it:g-ii} $\deg g_{st}(x) < \frac{1}{2} \rho_{st}$ for all $s < t$.
        \item \label{it:g-iii} $g^{\rev} = g \cdot \kappa$.
    \end{enumerate}
\end{theorem}

Following \cite[Section~2]{proudfoot-kls}, we will call $f$ (resp. $g$) \emph{the right (resp. left) Kazhdan--Lusztig--Stanley (KLS) function} associated to $\kappa$. If $P$ is bounded, we call $f_P(x):=f_{\widehat{0}\,\widehat{1}}(x)$ (resp. $g_P(x):=g_{\widehat{0}\,\widehat{1}}(x)$) the \emph{right (resp. left) Kazhdan--Lusztig--Stanley polynomial} of $P$.

For a detailed proof of the above theorem we refer to \cite[Theorem~2.2]{proudfoot-kls}. Furthermore, there is a converse to it proved in \cite[Theorem~2.5]{proudfoot-kls}, which guarantees that if $g$ is any element in $\mathcal{I}_{\rho}(P)$ satisfying the conditions \ref{it:g-i} and \ref{it:g-ii} then $\kappa:=g^{-1}g^{\rev}\in \mathcal{I}_{\rho}(P)$ is a $(P,\rho)$-kernel which has $g$ as its left KLS function. A completely analogous statement holds for right KLS functions.

A further important object in the KLS theory is the $Z$-function, studied in great detail by Proudfoot in \cite{proudfoot-kls}.

\begin{definition}\label{def:zeta-functions}
    Let $\kappa\in \mathcal{I}_{\rho}(P)$ be a $(P,\rho)$-kernel and let $f$ (resp. $g$) denote the right (resp. left) KLS function. The \emph{$Z$-function} associated to $\kappa$ is defined as the element $Z\in \mathcal{I}_{\rho}(P)$ given by
        \[ Z := g^{\rev} f = gf^{\rev}.\]
\end{definition}

The equality between the two expressions that define $Z$ can be seen directly from the equation $Z = g\kappa f$. The $Z$-function is symmetric, i.e., $Z^{\rev} = Z$ or, in other words, $Z_{st}(x) = x^{\rho_{st}}Z_{st}(x^{-1})$ for every $s\leq t$. Furthermore, it follows from \cite{proudfoot-kls} that if $\kappa$ is non-degenerate, then $\deg Z_{st}(x) = \deg \kappa_{st}(x)$ for all $s\leq t$.

Since the characteristic function is a $(P,\rho)$-kernel in any locally finite weakly ranked poset, we are led to consider the corresponding KLS functions.

\begin{example}\label{example:char-is-kernel}
    Let $\kappa = \chi$. By the preceding discussion, since $\chi = \upzeta^{-1}\cdot \upzeta^{\rev}$, we obtain that $g=\upzeta$ is the left KLS function. Trivially we have that $g$ is non-negative. In strong contrast, the right KLS function is a much more subtle and difficult object. Nevertheless, when $P$ is a geometric lattice and $\rho$ is the rank function, then the right KLS function is non-negative due to a deep result by Braden, Huh, Matherne, Proudfoot, and Wang \cite[Theorem~1.2]{braden-huh-matherne-proudfoot-wang}.  We note that when $P$ is not a geometric lattice, the right KLS function is not guaranteed to be non-negative (see Remark~\ref{rem:F-horrible}).
\end{example}

\begin{example}\label{example:ad-hoc-kernel}
    Let $P$ be the Boolean lattice having three atoms, and regard it as a graded poset so that $\rho_{st}$ is the length of any saturated chain starting at $s$ and ending at $t$. Fix any number $m\in \mathbb{Z}$ and define the following element $\kappa$ on $\mathcal{I}_{\rho}(P)$:
    \[
    \kappa_{st}(x) = \begin{cases}
        1 & \text{ if }\rho_{st} = 0,\\
        x-1 & \text{ if } \rho_{st} = 1, \\
        x^2-2x+1 & \text{ if }\rho_{st} = 2, \\
        x^3 + mx^2 -mx -1 & \text{ if }\rho_{st} = 3.
    \end{cases}
    %\qquad 
    %\overline{\kappa}_{st}(x) = \begin{cases}
    %    -1 & \text{ if }\rho_{st} = 0,\\
    %    1 & \text{ if }\rho_{st} = 1, \\
    %    x-1 &\text{ if } \rho_{st} = 2, \\
    %    x^2 +(m+1)x + 1 & \text{ if }\rho_{st} = 3.
    %\end{cases}
    \]
    
    A direct computation shows that $\kappa$ is a $(P,\rho)$-kernel.
    %The Chow polynomial of $P$ is then $\H_P(x) = x^2 +(m+7) x + 1$. Notice that by choosing any integer $m\leq - 8$ we can produce a Chow polynomial attaining a negative coefficient.
    Observe that the right KLS function is constant equal to $1$ on all proper intervals and the right KLS polynomial is $f_P(x) = 1 + (m+3)x$. Notice that if $m\leq -4$, the KLS function fails to be non-negative. Furthermore, the left KLS function $g$ equals $f$ in this specific case, i.e. $g_{st}=f_{st}$ for every $s \leq t$. (This is a coincidence, as the two functions may in general have very different behaviors.) The $Z$-function is equal to $(x+1)^{\rho_{st}}$ on all proper intervals and $Z_P(x) = x^3 + (m+6)x^2 + (m+6)x + 1$.
\end{example}

\begin{lemma}\label{lem:g-and-non-degeneracy-of-k}
    Let $\kappa$ be a $(P,\rho)$-kernel and let $f,g$ be the KLS functions. We have the following equalities of coefficients:
    \[
    [x^0]g_{st}(x) = [x^{\rho_{st}}]g_{st}^{\rev}(x) = [x^{\rho_{st}}]\kappa_{st}(x),
    \]
    \[
    [x^0]f_{st}(x) = [x^{\rho_{st}}]f_{st}^{\rev}(x) = [x^{\rho_{st}}]\kappa_{st}(x).
    \]
    In particular, the constant term of $f$ and $g$ is non-zero if and only if $\kappa$ is non-degenerate.
\end{lemma}

\begin{proof}
    The property of Theorem~\ref{thm:kls-functions}\ref{it:g-iii} translates into the following equation for any $s\leq t$:
    \[
    g_{st}^{\rev}(x) - g_{st}(x) = \kappa_{st}(x) + \sum_{s < w < t}g_{sw}(x)\kappa_{wt}(x).
    \]
    The proof of the statement follows from observing that every term in the sum on the right-hand side has degree smaller than $\rho_{st}$. The proof for $f$ is identical.
\end{proof}

\subsection{Graded posets, flag \texorpdfstring{$f$}{f}-vectors, and Cohen--Macaulayness}\label{subsec:cohen-macaulay}

Whenever $P$ is a finite graded bounded poset, the rank function $\rho\colon \Int(P)\to \mathbb{Z}_{\geq 0}$ can be computed as $\rho_{st} = \rho(t) - \rho(s)$ where $\rho(w)$ denotes the length of any saturated chain from $\widehat{0}$ to $w\in P$. Let us denote $r = \rho(\widehat{1})$ the rank of the poset $P$, and for each subset $S\subseteq \{0,1,\ldots,r\}$, say $S = \{s_1,\ldots,s_m\}$, define
    \[ \alpha_P(S) = \|\{\text{chains $w_1 < \cdots < w_m$ in $P$} : \rho(w_i) = s_i \text{ for $1\leq i\leq m$}\}\|.\]
The map $\alpha_P\colon 2^{\{0,\ldots,r\}}\to \mathbb{Z}_{\geq 0}$ is commonly known as the \emph{flag $f$-vector} of $P$. It can be encoded in an alternative way by considering instead the \emph{flag $h$-vector}, which is the map $\beta_P\colon 2^{\{0,\ldots,r\}}\to \mathbb{Z}$ defined by the condition
    \[ \alpha_P(S) = \sum_{T\subseteq S} \beta_P(T).\]
Let $\Delta$ be a simplicial complex, and denote by $f_i$ the number of faces of $\Delta$ having cardinality $i$ (or, equivalently, dimension $i-1$)\footnote{We are using the conventions of Bj\"orner in \cite{bjorner-matroids}.}. The \emph{$f$-vector} of $\Delta$ is defined by
    \[ f(\Delta) = (f_0, \ldots, f_{d})\]
where $d = \dim(\Delta) + 1$. The \emph{$f$-polynomial of $\Delta$} is the polynomial $f(\Delta,x) = f_0x^d + f_{1}x^{d-1} + \cdots + f_d$. The \emph{$h$-vector} and the \emph{$h$-polynomial} of $\Delta$ are defined via
\[h(\Delta,x) = h_0x^d + h_1 x^{d-1} + \cdots + h_d = f(\Delta,x-1).\]

Recall that to every poset $P$ we may associate a simplicial complex $\Delta(P)$, called the \emph{order complex of $P$}. The faces of $\Delta(P)$ correspond to chains of elements in $P$.  The \emph{$f$-vector} of the simplicial complex $\Delta(P)$ is encoded in the flag $f$-vector of $P$. Put precisely, 
    \[ f_{i}(\Delta(P)) = \sum_{\substack{S\subseteq [r-1]\\|S| = i}} \alpha_P(S).\]
It is not difficult to show that the $h$-vector of $\Delta(P)$ is given by
    \[ h_{i}(\Delta(P)) = \sum_{\substack{S\subseteq [r-1]\\|S| = i-1}} \beta_P(S).\]

Note that the flag $f$-vector of $P$ has non-negative values, as it count chains. However, the flag $h$-vector often fails to be non-negative. Similarly, the $h$-vector of $\Delta(P)$ can a priori have negative coefficients. In what follows we recapitulate an important case in which the flag $h$-vector of $P$ is indeed non-negative (and therefore the $h$-vector of $\Delta(P)$).

As any other simplicial complex, $\Delta(P)$ admits a geometric realization that we will denote $|\Delta(P)|$.
Note that every (open) interval $(s,t)=\{w\in P : s < w < t\}$ is itself a graded poset. By definition, we say that $P$ is \emph{Cohen--Macaulay} (over $\mathbb{Q}$) if the rational reduced homology groups of the order complex of every open interval $(s,t)$ satisfy
    \[ \widetilde{H}_i(\Delta(s,t)) = 0 \qquad \text{ for all $i \neq \rho_{st} - 2$}.\]
In other words, the (reduced) homology of every open interval $(s,t)$ must be concentrated in dimension $\dim \Delta(s,t) = \rho_{st}-2$. The class of Cohen--Macaulay posets comprises a number of well-studied families, such as geometric lattices, distributive lattices, posets that are EL-shellable, etc. It is worth noting that Cohen--Macaulayness is a topological property, that is, if $P_1$ and $P_2$ are posets and there is a homeomorphism $|\Delta(P_1)|\approx |\Delta(P_2)|$ then $P_1$ is Cohen--Macaulay if and only if $P_2$ is Cohen--Macaulay.

\begin{theorem}[{\cite[Theorem~3.3]{bjorner-garsia-stanley}}]
    Let $P$ be a finite graded bounded poset. If $P$ is Cohen--Macaulay, then the flag $h$-vector of $P$ has non-negative entries.
\end{theorem}

The non-negativity property in the above result follows from interpreting the entries of the flag $h$-vector as Betti numbers of rank-selected subposets of $P$.

\section{Chow functions}\label{sec:three}

In this section we will introduce the main objects of study in the present article and prove theorems in a general and abstract setting. In the later sections we will investigate the interactions of this framework with existing prior work. 

\subsection{Reduced kernels}

A well-known object in the theory of hyperplane arrangements is the ``reduced'' characteristic polynomial. If $\mathcal{A}$ denotes a non-empty central hyperplane arrangement, the characteristic polynomial $\chi_{\mathcal{A}}(x)$ vanishes when evaluated at $x = 1$. (This is an immediate fact that follows from the definition of the M\"obius function.) In this setting, and in the more general context of matroids, it is customary to define the \emph{reduced characteristic polynomial} by $\overline{\chi}_{\mathcal{A}}(x) = \frac{1}{x-1}\,\chi_{\mathcal{A}}(x)$. We point out that this agrees with the usual notation and conventions used, for instance, in \cite[Definition~9.1]{adiprasito-huh-katz}. 

In the same way that one is able to reduce the characteristic function by discarding the trivial zero $x = 1$, one can in fact reduce any $(P,\rho)$-kernel in the same way.

\begin{lemma}\label{lemma:divisibility-x-1}
    Let $\kappa$ be a $(P,\rho)$-kernel. Then, for every $s < t$ in $P$, the polynomial $\kappa_{st}(x)$ is divisible by $x - 1$. 
\end{lemma}

\begin{proof}
    We proceed by induction on the size of the interval $[s,t]$. Consider first the case in which the element $s$ is covered by $t$. The condition $\kappa \kappa^{\rev} = \delta$ is equivalent to
    \[
    0 = \kappa_{ss}(x)\kappa_{st}^{\rev}(x) + \kappa_{st}(x)\kappa_{tt}^{\rev}(x) =  \kappa_{st}^{\rev}(x) + \kappa_{st}(x).
    \]
    This implies that $\kappa_{st}(x) = -x^{\rho_{st}}\kappa_{st}(x^{-1})$. By evaluating both sides at $x=1$ we see that $\kappa_{st}(1) = -\kappa_{st}(1)$, which implies the desired property. Now, for the induction step, notice that 
    \[
    \kappa_{st}(x) + \kappa_{st}^{\rev}(x) = - \sum_{s < w < t} \kappa_{sw}(x)\kappa_{wt}^{\rev}(x).
    \]
    As all the polynomials on the right-hand side are associated to smaller intervals that are not singletons, evaluating at $x=1$ gives $\kappa_{st}(1) + \kappa_{st}^{\rev}(1) = 0$, which finishes the proof.
\end{proof}

\begin{definition}\label{def:reduced-kernel}
    Let $\kappa$ be a $(P,\rho)$-kernel. We define the corresponding \emph{reduced $(P,\rho)$-kernel} as the element $\overline{\kappa}\in \mathcal{I}_{\rho}(P)$ given by
    \[ \overline{\kappa}_{st}(x) 
    = \begin{cases}
        \frac{1}{x-1}\, \kappa_{st}(x) & \text{ if $s < t$}\\
        -1 & \text{ if $s = t$}.
    \end{cases}\]
\end{definition}

The choice we impose on defining $\kappa_{ss}(x)$ as $-1$ (as opposed to $1$) is intentional. An alternative approach would be to define $\overline{\kappa}_{st}(x)$ as $\frac{1}{1-x} \kappa_{st}(x)$ for $s < t$ and as $1$ for $s=t$, but this might cause confusion with the standard notation for the reduced characteristic polynomial. Since the latter is a key motivating example, we prefer to follow the convention in Definition~\ref{def:reduced-kernel}. On another note, we warn the reader that in other sources the notation $\overline{\kappa}$ stands for what we denote here as $\kappa^{\rev}$. We prefer to use the overline to denote ``reduced'' instead of ``reversed''. 

\subsection{Definition of Chow functions}

Since Proposition~\ref{prop:inverses-incidence-algebra} guarantees that a reduced $(P,\rho)$-kernel is invertible, we are motivated to consider the following notion, which constitutes the primary object of study in this article.

\begin{definition}\label{def:chow-function}
    Let $\kappa$ be a $(P,\rho)$-kernel. We define the \emph{Chow function} associated to $\kappa$, or the \emph{$\kappa$-Chow function}, as the element $\H\in \mathcal{I}_{\rho}(P)$ defined by
    \[ \H = - \left(\overline{\kappa}\right)^{-1}.\]
    If the poset $P$ is bounded, the polynomial $\H_P(x)=\H_{\widehat{0}\,\widehat{1}}(x)$ will be customarily called the \emph{$\kappa$-Chow polynomial} of the poset.
\end{definition}

As a consequence of having defined $\overline{\kappa}_{ss}(x)$ as $-1$ for every $s\in P$, the minus sign appearing in the above definition guarantees that $\H_{ss}(x) = 1$ for every $s\in P$. 
Furthermore, we note that Definition~\ref{def:chow-function} is equivalent to either of the following properties.  For all $s<t$ in $P$,
    \begin{align}
        \H_{st}(x) &= \sum_{s < w \leq t} \overline{\kappa}_{sw}(x)\, \H_{wt}(x)\qquad \text{ or, dually,} \label{eq:recurrence-chow}\\
        \H_{st}(x) &= \sum_{s\leq w < t} \H_{sw}(x)\, \overline{\kappa}_{wt}(x).\label{eq:recurrence-chow2}
    \end{align}

In the subsequent sections of this article we will focus our attention on a number of interesting examples of Chow functions.  

The following is the basic toolkit of properties that general Chow functions satisfy regarding degree and symmetry.

\begin{proposition}\label{prop:degree-and-symmetry}
    Let $\kappa$ be a $(P,\rho)$-kernel, and let $\H\in \mathcal{I}_{\rho}(P)$ be the corresponding Chow function. The  following properties hold.
    \begin{enumerate}[\normalfont(i)]
        \item\label{it:degree-of-H} For every $s < t$, we have that
            \[ [x^{\rho_{st}-1}] \H_{st}(x) = [x^{\rho_{st}}] \kappa_{st}(x).\]
            In particular, if $\kappa$ is non-degenerate, then $\deg \H_{st}(x) = \rho_{st} - 1$ for every $s < t$.
        \item \label{it:chow-symmetry}The Chow function is symmetric, i.e.,
            \[ \H_{st}(x) = x^{\rho_{st} - 1}\, \H_{st}(x^{-1}) \qquad \text{ for every $s < t$}.\]
    \end{enumerate}
\end{proposition}

\begin{proof}
    From equation~\eqref{eq:recurrence-chow} we can write
    \begin{equation}\label{eq:some-recursion}
    \H_{st}(x) = \overline{\kappa}_{st}(x) + \sum_{s < w < t} \overline{\kappa}_{sw}(x)\H_{wt}(x).
    \end{equation}
    We first prove both claims in the case in which $s$ is covered by $t$. Notice that the above sum simplifies to
    \[ \H_{st}(x) = \overline{\kappa}_{st}(x).\]
    In particular $[x^{\rho_{st}-1}] \H_{st}(x) = [x^{\rho_{st}-1}]\overline{\kappa}_{st}(x) = [x^{\rho_{st}}]\kappa_{st}(x)$. On the other hand since $s$ is covered by $t$, proceeding as in the proof of Lemma~\ref{lemma:divisibility-x-1} the condition $\kappa\kappa^{\rev} = \delta$ tells us that
        \[ \kappa_{st}^{\rev}(x) + \kappa_{st}(x) = 0,\]
    which in turn is equivalent to
        \[ 0 = x^{\rho_{st}}\kappa_{st}(x^{-1}) + \kappa_{st}(x) = x^{\rho_{st}} (x^{-1}-1) \overline{\kappa}_{st}(x^{-1}) + (x-1)\overline{\kappa}_{st}(x). \]
    Hence, since $\H_{st}(x) = \overline{\kappa}_{st}(x)$, the last equation reduces to
        \[ x^{\rho_{st}-1} (1-x) \H_{st}(x^{-1}) + (x-1) \H_{st}(x) = 0,\]
    which can be simplified to $x^{\rho_{st}-1} \H_{st}(x^{-1}) = \H_{st}(x)$.  Hence we have proved our claims under the assumption that $s$ is covered by $t$.
    
    Now, proceeding by induction on the size of the interval $[s,t]$, every term in the sum on the right-hand side of equation~\eqref{eq:some-recursion} has degree at most $\rho_{sw}-1 + \rho_{wt}-1 = \rho_{st}-2$. Therefore $[x^{\rho_{st}-1}]\H_{st}(x) = [x^{\rho_{st}-1}]\overline{\kappa}_{st}(x) = [x^{\rho_{st}}] \kappa_{st}(x)$. Now, to prove the second property, we multiply the formula of equation~\eqref{eq:some-recursion} by $x-1$ and subtract the polynomial $\kappa_{st}(x)$ from both sides:
    \begin{align}
        (x-1)\H_{st}(x) - \kappa_{st}(x) &= \sum_{s<w<t}\kappa_{sw}(x)\H_{wt}(x) \label{step-0}\\
        &= \sum_{s<w<t}\kappa_{sw}(x)x^{\rho_{wt}-1}\H_{wt}(x^{-1}) \label{step-1}\\
        &= \sum_{s<w<t} \left(-\sum_{s\leq u <w}\kappa_{su}(x)x^{\rho_{uw}}\kappa_{uw}(x^{-1}) \right) x^{\rho_{wt}-1}\H_{wt}(x^{-1}) \label{step-2}\\
        &= -\sum_{s\leq u<t}\kappa_{su}(x) x^{\rho_{ut}-1} \left(\sum_{u < w < t}\kappa_{uw}(x^{-1})\H_{wt}(x^{-1}) \right) \nonumber\\
        &= -\sum_{s\leq u<t}\kappa_{su}(x) x^{\rho_{ut}-1} \left((x^{-1}-1)\H_{ut}(x^{-1}) - \kappa_{ut}(x^{-1}) \right) \label{step-3}\\
        &=x^{-1} \sum_{s\leq u < t}\kappa_{su}(x)x^{\rho_{ut}}\kappa_{ut}(x^{-1}) \\
        &\qquad\qquad - \frac{1-x}{x}\sum_{s\leq u < t}\kappa_{su}(x)x^{\rho_{ut}-1}\H_{ut}(x^{-1})\nonumber \\
        &= -\frac{1}{x}\kappa_{st}(x) + \frac{x-1}{x}\sum_{s\leq u < t}\kappa_{su}(x)x^{\rho_{ut}-1}\H_{ut}(x^{-1}),\label{step-4}
    \end{align}
    where in \eqref{step-1} we used the induction hypothesis since $[w,t]$ is a strictly smaller interval, in \eqref{step-2} we used that $\kappa$ is a $(P,\rho)$-kernel, in \eqref{step-3} we used equation~\eqref{step-0} again but changing the variable $x$ by $x^{-1}$ and $w$ by $u$, while in \eqref{step-4} we used again that $\kappa$ is a $(P,\rho)$-kernel. Note that it is possible to simplify the equality between the left-hand side of~\eqref{step-0} and the right-hand side of~\eqref{step-4}:
    \[
    x\H_{st}(x) = \kappa_{st}(x) + \sum_{s\leq u < t}\kappa_{su}(x) x^{\rho_{ut}-1}\H_{ut}(x^{-1}).
    \]
    Using the induction hypothesis once more, but now on the intervals $[u,t]$ for $s < u$, we write $x^{\rho_{ut}-1}\H_{ut}(x^{-1}) = \H_{ut}(x)$ and get
    \[
    x\H_{st}(x)= \kappa_{st}(x) + x^{\rho_{st}-1}\H_{st}(x^{-1}) + \sum_{s<u< t}\kappa_{su}(x)\H_{ut}(x).
    \]
    To conclude, observe that adding $\H_{st}(x)$ to both sides in equation~\eqref{step-0} and combining the result with the last formula we obtained, it is possible to see that $\H_{st}(x) = x^{\rho_{st}-1}\H_{st}(x^{-1})$, and the proof is complete.
\end{proof}

\begin{example}\label{example:negative-chow-functions}
    Continuing with the poset in Example~\ref{example:ad-hoc-kernel}, the Chow polynomial of $P$ can be computed by hand, yielding $\H_P(x) = x^2 +(m+7) x + 1$. Notice that by choosing any integer $m \leq -7$ the Chow function fails to be unimodal, and by choosing $m\leq - 8$ we can produce a Chow polynomial attaining a negative coefficient.
\end{example}

The following proposition provides an alternative characterization of Chow functions, that shows that it fulfills \emph{simultaneously} the key properties  \ref{it:f-iii} and \ref{it:g-iii} of both the right and left KLS functions in Theorem~\ref{thm:kls-functions}, but dropping instead the assumption on having degree at most half the rank.

\begin{proposition}\label{prop:alt-characterization-chow}
Let $\kappa$ be a $(P,\rho)$-kernel. The Chow function $\H$ is the unique element in $\mathcal{I}_{\rho}(P)$ such that 
    \begin{enumerate}[\normalfont(i)]
        \item $\H_{ss}(x) = 1$ for all $s\in P$,
        \item For every $s<t$ the polynomial $\H_{st}(x)$ is symmetric, with center of symmetry $\frac{1}{2}(\rho_{st}-1)$.
        \item $\kappa \H = \H^{\rev}$ or $\H\kappa = \H^{\rev}$.
    \end{enumerate}
\end{proposition}

\begin{proof}
    If  $s<t$, by multiplying equation~\eqref{eq:recurrence-chow} by $x-1$, we have:
    \[ (x-1)\H_{st}(x) = \sum_{s < w \leq t} \kappa_{sw}(x) \H_{wt},\]
    which after adding $\H_{st}$ to both sides translates into
    \[ x\H_{st}(x) = \sum_{s\leq w \leq t} \kappa_{sw}(x) \H_{wt}(x) = (\kappa\,\H)_{st}(x).\]
    Applying Proposition~\ref{prop:degree-and-symmetry}\ref{it:chow-symmetry} we have that
    \[\H^{\rev}_{st}(x) = x^{\rho_{st}}\H_{st}(x^{-1}) = x\H_{st}(x) = (\kappa\,\H)_{st}(x),\]
    which says that $\H^{\rev} = \kappa\,\H$. On the other hand, by multiplying equation~\eqref{eq:recurrence-chow2} by $x-1$ and adding $\H_{st}(x)$
    \[ x\H_{st}(x) = \sum_{s\leq w \leq t} \H_{sw}(x) \kappa_{wt}(x) = \H\kappa.\]
    Repeating the reasoning of above, this proves that $\H\kappa = \H^{\rev}$. Now, from any element $\H$ satisfying the assumptions of the statement, we obtain that for every $s<t$ 
    \[
    x\H_{st}(x) = \sum_{s \leq w \leq t}\kappa_{sw}(x)\H_{wt}(x).
    \]
    By moving the term corresponding to $w=s$ to the left-hand side and dividing by $x-1$ we obtain that
    \[
    \H_{st}(x) = \sum_{s < w \leq t}\overline{\kappa}_{sw}(x)\H_{wt}(x),
    \]
    or, equivalently, that $\overline{\kappa}\H = \delta$. The uniqueness of the inverse in $\mathcal{I}(P)$ lets us conclude.
\end{proof}

Lastly, the next proposition can be seen as the Chow counterpart of \cite[Proposition~2.5]{proudfoot-kls} and \cite[Theorem~6.5]{stanley-local}.

\begin{proposition}
    Let $\H$ be an element of $\mathcal{I}_\rho(P)$ such that $\H_{ss}(x) = 1$ for all $s \in P$, and such that for every $s<t$ the polynomial $\H_{st}(x)$ is symmetric with center of symmetry $\frac{1}{2}(\rho_{st}-1)$. There exists a unique $(P, \rho)$-kernel $\kappa$ such that $\H$ is the associated Chow function. 
\end{proposition}
\begin{proof}
    We define $\kappa = \H^{\rev}\H^{-1}$. This is clearly a $(P,\rho)$-kernel (cf. the discussion below Definition~\ref{def:p-kernel}) and $\kappa \H = \H^{\rev}$, hence $\H$ is the $\kappa$-Chow function by Proposition~\ref{prop:alt-characterization-chow}. Now consider a different function $\widetilde{\kappa}$ also having $\H$ as its associated Chow function. We trivially have that $\widetilde{\kappa}_{ss}(x) = \kappa_{ss}(x)$ for each $s\in P$. For a non-trivial interval $[s,t]$, by expanding the convolutions $\widetilde{\kappa}\H = \H^{\rev}$ and $\kappa\H = \H^{\rev}$ and by proceeding by induction on the size of the intervals using the assumption that $\widetilde{\kappa}_{s't'}(x) = \kappa_{s't'}(x)$ on all intervals $[s',t']$ that are smaller in size than $[s,t]$, it is straightforward to conclude that $\widetilde{\kappa}_{st}(x) = \kappa_{st}(x)$, and thus $\widetilde{\kappa} = \kappa$ in $\mathcal{I}_{\rho}(P)$.
\end{proof}

\subsection{The relation between KLS and Chow functions}

Our goal now is to establish a general set of formulas that link the KLS functions with the Chow functions. Although these formulas are valid at a great level of generality, even specific examples of them are highly non-trivial when specializing to some of the concrete examples mentioned in the subsequent sections. We start with a preparatory lemma. Throughout the rest of this section, we assume that $\kappa$ denotes a $(P,\rho)$-kernel, $\H$ denotes the Chow function, and $f$ (resp. $g$) denotes the right (resp. left) KLS function.

\begin{lemma}\label{lem:kbarf-and-gkbar}
	The products in $\mathcal{I}_{\rho}(P)$ between the KLS functions and the reduced $(P,\rho)$-kernel are given by 
	\[(\overline{\kappa} f)_{st}(x) =  \begin{cases}
		\dfrac{f^{\rev}_{st}(x) - xf_{st}(x)}{x-1} & \text{ if $ s < t$, }\\
		-1 & \text{ if $s=t$},
	\end{cases} \qquad (g\overline{\kappa})_{st}(x) = \begin{cases}
		\dfrac{g^{\rev}_{st}(x) - xg_{st}(x)}{x-1} & \text{ if $ s < t$,}\\
		-1 & \text{ if $s=t$}.
	\end{cases}\]
\end{lemma}

\begin{proof}
	We will only do the proof for $g$ since the one for $f$ is very similar. By Theorem~\ref{thm:kls-functions}\ref{it:g-iii} we have that $g\kappa = g^{\rev}$. In particular, for every $s\leq t$ we have
		\[g^{\rev}_{st}(x) = \sum_{s\leq w\leq t} g_{sw}(x)\, \kappa_{wt}(x).\]
	Subtracting on both sides the term on the right corresponding to $w = t$, we obtain
		\[
		g^{\rev}_{st}(x)  - g_{st}(x) = \sum_{s\leq w < t} g_{sw}(x)\, \kappa_{wt}(x),\\
		= (x - 1) \sum_{s\leq w < t} g_{sw}(x)\, \overline{\kappa}_{wt}(x).
		\]
	Hence, it follows that
		\[
			(g\overline{\kappa})_{st}(x) = \sum_{s\leq w \leq t} g_{sw}(x)\, \overline{\kappa}_{wt}(x)\\
			= \tfrac{1}{x-1} \left(g^{\rev}_{st}(x)  - g_{st}(x)\right) - g_{st}(x) = \frac{g^{\rev}_{st}(x)  - xg_{st}(x)}{x-1},
		\]
		as desired.
\end{proof}

We are now ready to state and prove one of the most important tools we will employ throughout the rest of the paper. We will refer to the following formulas as ``numerical canonical decompositions'' because, as we will argue in Section~\ref{subsec:hodge-theory}, they are numerical shadows of deep algebro-geometric properties of certain cohomologies.

\begin{theorem}[Numerical canonical decomposition]\label{thm:ncd}
    The Chow function can be computed from the KLS functions as follows:
    \begin{align}
        \H_{st}(x) &= \frac{f^{\rev}_{st}(x) - f_{st}(x)}{x-1} + \sum_{s < w < t}\H_{sw}(x) \,\frac{f^{\rev}_{wt}(x) - x f_{wt}(x)}{x-1} ,\label{eq:ncd-f}\\
        \H_{st}(x) &= \frac{g^{\rev}_{st}(x) - g_{st}(x)}{x-1} + \sum_{s < w < t}\frac{g^{\rev}_{sw}(x) - x g_{sw}(x)}{x-1}\, \H_{wt}(x).\label{eq:ncd-g}
    \end{align}
\end{theorem}

\begin{proof}
    We only show the second one, as the proof of the first is entirely analogous. We write 
    \[ g = g\delta = g(-\overline\kappa)\H = -(g\overline{\kappa})\H.
    \]
    Now, from Lemma~\ref{lem:kbarf-and-gkbar} we can compute $g\overline{\kappa}$, and convolving this with $\H$ gives us
    \begin{align*}
        g_{st}(x) &= -\left(-\H_{st}(x) + \sum_{s < w \leq t} \frac{g_{sw}^{\rev}(x) - x g_{sw}(x)}{x-1} \H_{wt}(x)\right).
    \end{align*}
    Now we rearrange the terms and separate one summand from the sum on the right to obtain~\eqref{eq:ncd-g}, as desired:
    \begin{align}
    \H_{st}(x) &= g_{st}(x) + \sum_{s < w \leq t}\frac{g^{\rev}_{sw}(x) - x g_{sw}(x)}{x-1}\, \H_{wt}(x)\label{eq:old-ncd}\\
    &= g_{st}(x) + \frac{g^{\rev}_{st}(x) - x g_{st}(x)}{x-1} + \sum_{s < w < t}\frac{g^{\rev}_{sw}(x) - x g_{sw}(x)}{x-1}\, \H_{wt}(x)\nonumber\\
    &= \frac{g^{\rev}_{st}(x) - g_{st}(x)}{x-1} + \sum_{s < w < t}\frac{g^{\rev}_{sw}(x) - x g_{sw}(x)}{x-1}\, \H_{wt}(x).\qedhere
    \end{align}
\end{proof}

Although the numerical canonical decompositions are fundamental throughout this paper, in some cases we will need to write the Chow function in a non-recursive way. The following provides us with a non-recursive formula, which is given by a sum over chains.

\begin{theorem}\label{thm:chow-from-kl}
    Let $\kappa$ be a $P$-kernel, let $f$ and $g$ be the right and left KLS functions, and let $\H$ be the Chow function. Then,
    \begin{align}
        \H_{st}(x) &= \sum_{s\leq p_0 < p_1 < \cdots < p_m = t} f_{sp_0}(x) \prod_{i=1}^m \frac{f^{\rev}_{p_{i-1}p_i}(x) - xf_{p_{i-1}p_i}(x)}{x-1}\\
        &=\sum_{s = p_0 < p_1 < \cdots < p_m\leq t} \left(\prod_{i=1}^m\frac{g^{\rev}_{p_{i-1}p_i}(x) - xg_{p_{i-1}p_i}(x)}{x-1}\right)g_{p_m\,t}(x).\label{eq:chow-from-kl-on-chains}
    \end{align}
\end{theorem}

\begin{proof}
    Once more, we will write the proof for the second of the two formulas, as the first is analogous. Note that by equation~\eqref{eq:old-ncd}, we can write 
    \[ \H_{st}(x) = g_{st}(x) + \sum_{s < p_1 \leq t}\frac{g^{\rev}_{sp_1}(x) - x g_{sp_1}(x)}{x-1}\, \H_{p_1t}(x).\]
    In turn, the polynomial $\H_{p_1t}(x)$ can be computed by the same recursion. That is,
    \begin{multline*}
        \H_{st}(x) = g_{st}(x) + \\
        \sum_{s < p_1 \leq t}\frac{g^{\rev}_{sp_1}(x) - x g_{sp_1}(x)}{x-1}\,\left(g_{p_1t}(x) + \sum_{p_1 < p_2 \leq t} \frac{g^{\rev}_{p_1p_2}(x) - x g_{p_1p_2}(x)}{x-1}\H_{p_2t}(x)\right) .
    \end{multline*}
    Iterating this, we obtain the formula claimed in equation~\eqref{eq:chow-from-kl-on-chains}.
\end{proof}

\subsection{Non-negativity and unimodality of Chow functions}

It is known that the right and left KLS functions arising from a $(P,\rho)$-kernel $\kappa$ can fail to be non-negative. Similarly, by our Example~\ref{example:negative-chow-functions} a Chow function can fail to be non-negative too. As we will show now, there is a striking connection between the non-negativity of the KLS functions and the non-negativity and unimodality of the Chow function. The proof will need the following preparatory lemma.

\begin{lemma}\label{lemma:prod-unimodals}
    Let $p(x)$ and $q(x)$ be polynomials that are non-negative, symmetric, and unimodal. Then $p(x)q(x)$ is non-negative, symmetric, and unimodal.
\end{lemma}

The proof of the above lemma can be found in \cite[Proposition~1]{stanley-unimodality}. We mention that the assumption on the symmetry is essential, as in general the product of two non-negative unimodal polynomials may fail to be unimodal.

\begin{theorem}\label{thm:kls-positive-chow-unimodal}
    Let $\kappa$ be a $(P,\rho)$-kernel. If either the right KLS function $f$ or the left KLS function $g$ are non-negative, then the Chow function $\H$ is non-negative and unimodal.
\end{theorem}

\begin{proof}
    We will only indicate the proof for $g$, as the one for $f$ is very similar. Let us assume that the left KLS function $g$ is non-negative. We shall proceed by induction on $\rho_{st}$. By definition, when $\rho_{st}=0$, we have $s = t$, and $\H_{st}(x)=1$ which is non-negative and unimodal. Let us assume that $\H_{st}(x)$ is unimodal whenever $\rho_{st} \leq \ell$, and consider any interval $[s,t]$ such that $\rho_{st} = \ell + 1$. We will use the numerical canonical decomposition.
    We claim that the sum in~\eqref{eq:ncd-g} consists of symmetric unimodal polynomials all of which have center of symmetry $\frac{1}{2}(\rho_{st}-1)$. We rely crucially on the following two ingredients: i) the fact that $\deg g_{st} < \frac{1}{2}\rho_{st}$ for all $s < t$, and ii) the assumption that $g_{st}(x)$ has non-negative coefficients. For concreteness, let us write $g_{st}(x) = g_0 + g_1 \,x +\cdots + g_{d} \,x^d$, where $d = \lfloor \frac{\rho_{st}-1}{2}\rfloor$, and each $g_i\geq 0$. We have that
        \begin{align*}
            g^{\rev}_{st}(x) - g_{st}(x) &= x^{\rho_{st}} g_{st}(x^{-1}) - g_{st}(x)\\ 
            &= g_0\,x^{\rho_{st}} + g_1\, x^{\rho_{st}-1} + \cdots + g_d x^{\rho_{st}-d} - g_d\,x^{d} - \cdots - g_1\,x - g_0.
        \end{align*}
    Notice that $2d < \rho_{st}$ implies that $\rho_{st} - d > d$, so that each of the monomials appearing above has a different exponent. One can group the terms of degree $i$ and $\rho_{st} - i$ for $i=0,\ldots, d$ obtaining
        \[
            g^{\rev}_{st}(x) - g_{st}(x) = \sum_{i=0}^{d} g_i \left(x^{\rho_{st}-i} - x^i\right)\\
            = (x-1)\sum_{i=0}^{d} g_i \,x^i\,\left(1 + x + \cdots + x^{\rho_{st}-1-2i}\right).
        \]
    In particular, the non-negativity of the $g_i$'s yields that the polynomial $\frac{g^{\rev}_{st}(x) - g_{st}(x)}{x-1}$ is non-negative, symmetric, and unimodal, having center of symmetry $\frac{1}{2}(\rho_{st}-1)$. Similarly, for each $s < w < t$, let us write $g_{sw}(x) = g_0+\cdots + g_d\, x^d$, with each $g_i\geq 0$ and $d = \lfloor\frac{\rho_{sw}-1}{2}\rfloor$. Notice that even if we subtract $xg_{sw}(x)$ instead of just $g_{sw}(x)$ as in the previous computation, after possibly a single cancellation all the monomials have different exponents in $g^{\rev}_{sw}(x) - xg_{sw}(x)$. Therefore, we can compute
    \[g^{\rev}_{sw}(x) - x g_{sw}(x) = x\,(x-1) \sum_{i=0}^d g_i\, x^i \left(1+\cdots + x^{\rho_{sw}-2i}\right),\]
    which gives that $\frac{g^{\rev}_{sw}(x) - x g_{sw}(x)}{x-1}$ is non-negative, unimodal, and symmetric with center of symmetry $\frac{1}{2}\rho_{sw}$. 
    
    Now, the induction hypothesis guarantees that the polynomials $\H_{wt}(x)$ for $s < w < t$ are unimodal because $\rho_{wt} \leq \rho_{st} - 1 = \ell$. Furthermore, by Proposition~\ref{prop:degree-and-symmetry}, the center of symmetry is $\frac{\rho_{wt}-1}{2}$. On the other hand, by Lemma~\ref{lemma:prod-unimodals} the product 
    \[\frac{g^{\rev}_{sw}(x) - x g_{sw}(x)}{x-1}\, \H_{wt}(x)\] is non-negative, symmetric, and unimodal, and its center of symmetry is $\frac{1}{2}\rho_{sw} + \frac{1}{2}(\rho_{wt}-1) = \frac{1}{2}(\rho_{st} -1)$.
    Since we have proved that all the summands appearing are non-negative, symmetric, unimodal, and share a common center of symmetry, it follows that  $\H_{st}(x)$ fulfills the same property. By induction, the result follows.
\end{proof}

\subsection{Augmented Chow functions}\label{sec:augmented}

The goal in this section is to introduce ``augmented'' counterparts of the Chow function. 

\begin{definition}
    Let $\kappa$ be a $(P,\rho)$-kernel. Consider the following two elements of $\mathcal{I}_{\rho}(P)$:
        \begin{align*}
            F &= \H \cdot f^{\rev},\\
            G &= g^{\rev}\cdot \H.
        \end{align*}
    We call $F$ (resp. $G$) \emph{the right (resp. left) augmented Chow function associated to $\kappa$}.
\end{definition}

As was mentioned earlier, with KLS functions it is important to make the distinction between left and right. In some cases it may happen that these two functions exhibit a striking difference in their complexity (cf. Example~\ref{example:char-is-kernel}). Likewise, the augmented Chow functions $F$ and $G$ can behave in different ways (in contrast to the Chow functions, where there is no distinction between left and right). In the next section we will see this phenomenon in the context of $\kappa = \chi$, i.e., augmented characteristic Chow functions.

\begin{example}
    Going back once more to the poset in Example~\ref{example:ad-hoc-kernel}, the left and right augmented Chow polynomials happen to coincide, and they are equal to
    \[
    G_P(x) = x^3 + (m+10)x^2 + (m+10)x + 1.
    \]
\end{example}

We have the following augmented counterpart for Proposition~\ref{prop:degree-and-symmetry}.

\begin{proposition}\label{prop:augmented-degree-symmetry}
    Let $\kappa$ be a $(P,\rho)$-kernel, and let $F,G\in \mathcal{I}_{\rho}(P)$ be the corresponding augmented Chow functions. Then, the following properties hold:
    \begin{enumerate}[\normalfont(i)]
        \item For every $s \leq t$, we have that
            \[ [x^{\rho_{st}}] F_{st}(x) = [x^{\rho_{st}}] G_{st}(x) = [x^{\rho_{st}}] \kappa_{st}(x).\]
            In particular, if $\kappa$ is non-degenerate, we have that $\deg F_{st} = \deg G_{st} = \rho_{st}$ for every $s \leq t$. 
        \item The augmented Chow functions are symmetric, i.e.,
            \begin{align*} 
            F_{st}(x) &= x^{\rho_{st}}\, F_{st}(x^{-1}),\\
            G_{st}(x) &= x^{\rho_{st}}\, G_{st}(x^{-1}),
            \end{align*}
        for every $s\leq t$. In other words, $F^{\rev} = F$ and $G^{\rev} = G$.
    \end{enumerate}
\end{proposition}

\begin{proof}
    We prove the statement for $G$, as the proof for $F$ is analogous. The equation $G = g^{\rev}\,\H$ translates into the equation
    \[
    G_{st}(x) = g_{st}^{\rev}(x) + \sum_{s \leq w < t} g_{sw}^{\rev}(x) \H_{wt}(x) \qquad \text{for every $s\leq t$,}
    \]
    (where in the right-hand side we isolated one summand of the convolution).
    Every term in the sum has degree at most $\rho_{sw} + \rho_{wt} - 1 = \rho_{st}-1$, and therefore $[x^{\rho_{st}}]G_{st}(x) = [x^{\rho_{st}}]g_{st}^{\rev}(x) = [x^{\rho_{st}}]\kappa_{st}(x)$, where the last equality follows from Lemma~\ref{lem:g-and-non-degeneracy-of-k}. To prove the second statement we recall that Proposition~\ref{prop:alt-characterization-chow} guarantees that $\kappa\H = \H^{\rev}$, and thus
    \[G = g^{\rev} \H = (g \kappa) \H = g (\kappa \H) = g \H^{\rev} = \left( g^{\rev} \H \right)^{\rev} = G^{\rev}.\qedhere\]
\end{proof}

The following result provides an augmented analogue for the numerical canonical decomposition of Theorem~\ref{thm:ncd}.

\begin{theorem}[Augmented numerical canonical decomposition]\label{thm:aug-ncd}
    The augmented Chow functions can be computed from the $Z$-function and the KLS functions as follows:
    \begin{align}
    F_{st}(x) &= Z_{st}(x) + \sum_{s < w \leq t} \frac{g_{sw}^{\rev}(x) - xg_{sw}(x)}{x-1} \, F_{wt}(x),\label{eq:ncd-F}\\
    G_{st}(x) &= Z_{st}(x) + \sum_{s \leq w < t} G_{sw}(x) \, \frac{f_{wt}^{\rev}(x) - xf_{wt}(x)}{x-1}. \label{eq:ncd-G}
    \end{align}
\end{theorem}

\begin{proof}
    We prove the statement only for $G$ as the proof for $F$ is analogous. We know that \[
    Z = g^{\rev}f = -g^{\rev}(\H\overline{\kappa})f = -(g^{\rev}\H)\overline{\kappa}f = -G(\overline{\kappa}f).\]
    By virtue of Lemma~\ref{lem:kbarf-and-gkbar}, isolating one term of the convolution on the right-hand side of the above display, for every $s \leq t$ we have
    \[ Z_{st}(x) = -\left( -G_{st}(x) + \sum_{s\leq w< t} G_{sw}(x) \frac{f_{wt}^{\rev}(x) - xf_{wt}(x)}{x-1}\right),\]
    which after a rearrangement of the terms yields a proof of equation~\eqref{eq:ncd-G}.
\end{proof}

Now we have the tools to state and prove a result relating properties of the KLS and $Z$-functions to properties of the augmented Chow functions. 

\begin{theorem}
    Let $\kappa$ be a $(P,\rho)$-kernel, and let $F$ (resp. $G$) be the right (resp. left) augmented Chow functions. The following hold:
    \begin{enumerate}[\normalfont(i)]
        \item If $f$ (resp. $g$) is non-negative, then $F$ (resp. $G$) is non-negative.
        \item If $Z$ is non-negative and unimodal and $g$ (resp. $f$) is non-negative, then $F$ (resp. $G$) is unimodal.
    \end{enumerate}
\end{theorem}

\begin{proof}
    We do the proof for $F$, since the proof for $G$ is almost identical. The non-negativity of $f$ implies the non-negativity of $\H$ via Theorem~\ref{thm:kls-positive-chow-unimodal}, so the convolution $F = \H\cdot f^{\rev}$ is obviously non-negative. This proves the first property. 

    Assuming that $g$ is non-negative, from the proof of Theorem~\ref{thm:kls-positive-chow-unimodal} we know that the element
    \[\frac{g_{sw}^{\rev}(x) - xg_{sw}(x)}{x-1}\] is non-negative, symmetric, and unimodal, and its center of symmetry is $\frac{1}{2}(\rho_{sw}-1)$. In particular, we can induct using the augmented numerical canonical decomposition. Assuming that $F$ is unimodal on all proper intervals of $[s,t]$, we have that 
        \[ \sum_{s < w \leq t} \frac{g_{sw}^{\rev}(x) - xg_{sw}(x)}{x-1} \, F_{wt}(x)\]
    is a sum of unimodal and symmetric polynomials having the same center of symmetry. Observe that $Z$ also has that center of symmetry, so equation~\eqref{eq:ncd-F} gives that $F_{st}(x)$ is unimodal.
\end{proof}

\begin{remark}
    We do not know if it is possible to remove some of the assumptions for the second part of the prior statement. We have tried to construct examples showing that the requirements are all essential, but we could not find any. In particular, it would be very interesting to have an example in which $f$ and $g$ are non-negative but $Z$ is not unimodal. 
\end{remark}

\begin{remark}
    It is possible to define $F$ and $G$ without making explicit reference to the Chow function. Consider the elements $F^{\perp},G^{\perp} \in \mathcal{I}_\rho(P)$ defined by
    \[
    F^{\perp}_{st}(x) := \frac{x(f^{-1})^{\rev}_{st}(x) - f_{st}^{-1}(x)}{x-1}, \qquad G^{\perp}_{st}(x) := \frac{x(g^{-1})^{\rev}_{st}(x) - g_{st}^{-1}(x)}{x-1}.
    \]
    Notice that these are defined in terms of the inverses of the right and left KLS functions. It can be proved that
     \[F = \left(F^{\perp}\right)^{-1}, \qquad G = \left(G^{\perp}\right)^{-1}. \]
    Or, equivalently, that $F^{\perp} = F^{-1}$ and $G^{\perp} = G^{-1}$. Since we will not need these formulas in the remainder of the paper, we omit the proof.
\end{remark}

\section{Characteristic Chow functions of graded posets and geometric lattices}\label{sec:four}

In this section, we will study the properties of the Chow function that arises from the characteristic function in a finite graded bounded poset $P$. Under these assumptions, as was explained in Section~\ref{subsec:cohen-macaulay}, the rank function $\rho_{st}$ is given by the length of an arbitrary saturated chain from $s$ to $t$ or, equivalently, $\rho_{st} = \rho(t) - \rho(s)$, where $\rho(w)$ stands for the length of an arbitrary saturated chain from $\widehat{0}$ to $w$.

We refer to the Chow functions arising from $\chi$ as \emph{characteristic Chow functions} or \emph{$\chi$-Chow functions}. We will first establish a number of general properties that $\chi$-Chow polynomials of finite graded bounded posets satisfy, and later we will explain the consequences for the central example of matroids.

We start by noting explicitly that the characteristic function in a graded poset is combinatorially invariant. This readily implies that the KLS functions, the Chow function, and the augmented Chow functions are combinatorially invariant as well. 

Some properties of posets (e.g. being graded, or being Cohen--Macaulay) are \emph{hereditary on closed intervals}, that is, if $P$ satisfies them, so do all the closed intervals of $P$. By definition a family $\mathcal{C}$ of (isomorphism classes of) posets is said to be \emph{hereditary} if $P\in \mathcal{C}$ implies that all the closed intervals of $P$ lie in $\mathcal{C}$.

As a consequence of the above paragraph, if we prove a theorem (e.g. positivity) about all Chow polynomials of posets belonging to a hereditary class of posets, the same theorem will be true for the Chow functions of these posets. Therefore, we will often write our statements referring only to Chow polynomials of bounded posets, understanding that they carry over verbatim to Chow functions. The entirety of the preceding discussion also applies to KLS polynomials and augmented Chow polynomials.

\subsection{Basic properties and examples}

As we explained in Example~\ref{example:char-is-kernel}, when $\kappa = \chi$, the left KLS function is $g=\upzeta$. In particular, for every $s\leq t$ we have $g^{\rev}_{st}(x) = x^{\rho_{st}}$. By plugging this into the second of the two formulas in Theorem~\ref{thm:chow-from-kl}, we obtain the following result.

\begin{theorem}\label{thm:chi-chow-graded}
    Let $P$ be a finite graded bounded poset. The $\chi$-Chow polynomial of $P$ is given by
    \[ \H_P(x) = \sum_{\widehat{0} = p_0 < p_1 < \cdots < p_m\leq \widehat{1}}\; \prod_{i=1}^m\frac{x(x^{\rho(p_{i}) - \rho(p_{i-1})-1} \; - 1)}{x-1}.\]
\end{theorem}

\begin{example}
    If $P=C_n$ is a chain on $n\geq 2$ elements, the $\chi$-Chow polynomial of $P$ is given by
        \[ \H_P(x) = (x+1)^{n-2}.\]
    The above identity can be easily proved by induction. On the other hand, if $P=B_n$ is a Boolean lattice on $n\geq 1$ atoms, the Chow polynomial of $P$ is
        \[ \H_P(x) = A_n(x),\]
    the $n$-th \emph{Eulerian polynomial}, which has $i$-th coefficient equal to the number of permutations $\sigma\in \mathfrak{S}_n$ having exactly $i$ descents. This can be proved directly by induction, or  via Theorem~\ref{thm:chow-matroid} below (since Boolean lattices are geometric). Note that $\chi$-Chow polynomials behave erratically under Cartesian product: one can observe this by noting that $B_n$ is the $n$-fold Cartesian product of $B_1$ with itself.
\end{example}

\begin{example}\label{ex:some-chow-polys}
    Consider the graded posets $P$ (on the left) and $Q$ (on the right) depicted in Figure~\ref{fig:poset1}. Neither of these posets is Cohen--Macaulay. This is apparent for $P$ since the flag $h$-vector has $\beta_P(\{3,4,6\}) = -1$, whereas for $Q$ the flag $h$-vector is non-negative (but Cohen--Macaulayness still fails). Using the characteristic function as kernel, the left and right KLS polynomials of $P$ and $Q$ are given by
    \begin{align*}
    f_P(x) &= x^2 + 1, & g_P(x) &= 1,\\
    f_Q(x) &= 1, & g_Q(x) &= 1.
    \end{align*}
    The Chow polynomials can be calculated via the formula of Theorem~\ref{thm:chi-chow-graded}, yielding 
    \begin{align*}
        \H_P(x) &= x^5 + 8x^4 + 20x^3 + 20 x^2 + 8 x + 1,\\
        \H_Q(x) &= x^4 + 13x^3 + 25 x^2 + 13 x + 1.
    \end{align*}
    
    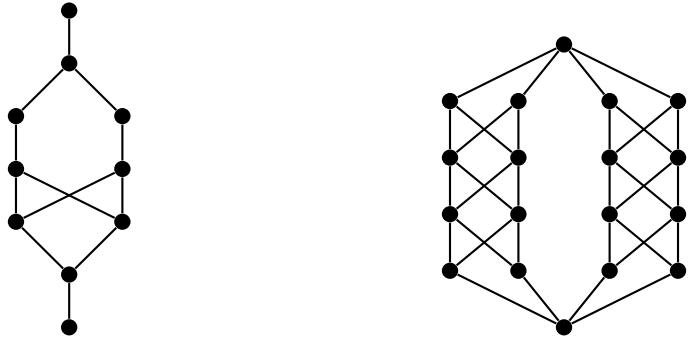
\begin{figure}[ht]
    \centering
	\begin{tikzpicture}  
	[scale=0.7,auto=center,every node/.style={circle,scale=0.8, fill=black, inner sep=2.7pt}] 
	\tikzstyle{edges} = [thick];
	
	\node[] (a1) at (0,0) {};  
	\node[] (a2) at (0,1)  {};  
	\node[] (a3) at (-1,2) {};
	\node[] (a4) at (1,2) {};
	\node[] (a5) at (-1,3)  {};  
	\node[] (a6) at (1,3)  {};  
	\node[] (a7) at (-1,4)  {};  
	\node[] (a8) at (1,4)  {};  
	\node[] (a9) at (0,5) {};
        \node[] (a10) at (0,6) {};
	
	\draw[edges] (a1) -- (a2);  
	\draw[edges] (a2) -- (a3);  
	\draw[edges] (a2) -- (a4);
	\draw[edges] (a3) -- (a5);
        \draw[edges] (a3) -- (a6);
	\draw[edges] (a4) -- (a6);
	\draw[edges] (a4) -- (a5);
	\draw[edges] (a5) -- (a7);
	\draw[edges] (a6) -- (a8);
	\draw[edges] (a7) -- (a9);
	\draw[edges] (a8) -- (a9);
	\draw[edges] (a9) -- (a10);
	\end{tikzpicture}\hspace*{4cm}
        \begin{tikzpicture}  
	[scale=0.75,auto=center,every node/.style={circle,scale=0.8, fill=black, inner sep=2.7pt}] 
	\tikzstyle{edges} = [thick];
	
	\node[] (cero) at (0,0) {};  
	\node[] (a1) at (0.8,1)  {};  
	\node[] (a2) at (2,1) {};
	\node[] (a3) at (0.8,2) {};
	\node[] (a4) at (2,2)  {};  
	\node[] (a5) at (0.8,3)  {};  
	\node[] (a6) at (2,3)  {};  
	\node[] (a7) at (0.8,4)  {};  
	\node[] (a8) at (2,4) {};
        \node[] (uno) at (0,5) {};
        \node[] (b1) at (-0.8,1)  {};  
	\node[] (b2) at (-2,1) {};
	\node[] (b3) at (-0.8,2) {};
	\node[] (b4) at (-2,2)  {};  
	\node[] (b5) at (-0.8,3)  {};  
	\node[] (b6) at (-2,3)  {};  
	\node[] (b7) at (-0.8,4)  {};  
	\node[] (b8) at (-2,4) {};

	\draw[edges] (a1) -- (a3);  
	\draw[edges] (a2) -- (a3);
	\draw[edges] (a1) -- (a4);
	\draw[edges] (a2) -- (a4);

        \draw[edges] (a3) -- (a5);  
	\draw[edges] (a4) -- (a6);
	\draw[edges] (a3) -- (a6);
	\draw[edges] (a4) -- (a5);

        \draw[edges] (a5) -- (a7);  
	\draw[edges] (a6) -- (a8);
	\draw[edges] (a5) -- (a8);
	\draw[edges] (a6) -- (a7);

        \draw[edges] (b1) -- (b3);  
	\draw[edges] (b2) -- (b3);
	\draw[edges] (b1) -- (b4);
	\draw[edges] (b2) -- (b4);

        \draw[edges] (b3) -- (b5);  
	\draw[edges] (b4) -- (b6);
	\draw[edges] (b3) -- (b6);
	\draw[edges] (b4) -- (b5);

        \draw[edges] (b5) -- (b7);  
	\draw[edges] (b6) -- (b8);
	\draw[edges] (b5) -- (b8);
	\draw[edges] (b6) -- (b7);

        \draw[edges] (cero) -- (a1);  
	\draw[edges] (cero) -- (a2);
        \draw[edges] (cero) -- (b1);  
	\draw[edges] (cero) -- (b2);
        \draw[edges] (uno) -- (a7);  
	\draw[edges] (uno) -- (a8);
        \draw[edges] (uno) -- (b7);  
	\draw[edges] (uno) -- (b8);
	
	\end{tikzpicture}\caption{The posets $P$ and $Q$ in Example~\ref{ex:some-chow-polys}.}\label{fig:poset1}
\end{figure}
\end{example}

\subsection{Combinatorics of (left) augmented Chow polynomials}

Now we investigate the augmented Chow polynomials arising from this setting. The left augmented Chow polynomial admits a beautiful combinatorial description (see Corollary~\ref{coro:aug-from-tail}), while the right augmented Chow polynomial is much more complicated to understand.

\begin{definition}
    Let $P$ be any poset. We define the \emph{augmentation} of $P$, denoted $\operatorname{aug}(P)$, as the poset obtained by adding a minimum element to $P$. 
\end{definition}

At the level of the topology of the order complexes, $|\Delta(\aug(P))|$ is homeomorphic to a cone over $|\Delta(P)|$. We also note that if $P$ already has a minimum element $\widehat{0}$, then $\widehat{0}$ becomes an atom of $\operatorname{aug}(P)$. The poset $P$ of Figure~\ref{fig:poset1} is the augmentation of a poset that already had a minimum. 

Recall that whenever we have two posets $P$ and $Q$, their \emph{ordinal sum} $P\oplus Q$ is defined as the poset on $P\sqcup Q$, preserving the ordering relations of both $P$ and $Q$, and imposing that $s\leq t$ for every $s\in P$ and $t\in Q$. For two graded bounded posets $P$ and $Q$, we define a similar operation that we will call the \emph{join} of $P$ and $Q$ and denote by $P*Q$. Precisely, we define $P*Q = P\oplus (Q\smallsetminus \{\widehat{0}_Q\})$. Note that $P*Q$ also equals $(P\smallsetminus \{\widehat{1}_P\})\oplus Q$. We point out that in other sources the join is defined in a slightly different way (see, e.g., \cite[p.~485]{stanley-cd}). Also, it is clear from the definitions that $\aug(P) \cong C_2 * P$, where $C_2$ is a chain on two elements.

\begin{proposition}\label{prop:chow-of-join}
    Let $P$ and $Q$ be two graded bounded posets. The following formula for the characteristic Chow function of $P*Q$ holds:
    \[ \H_{P*Q}(x) = \H_P(x)\cdot G_Q(x) = \H_P(x) \cdot \H_{\aug(Q)}(x).\]
\end{proposition}

\begin{proof}
    The key observation is that for every $s\in P$ and $t \in Q$, we have that \[
    \chi_{st}(x) = \chi_{s,\widehat{1}_P}(x)\upzeta^{\rev}_{\widehat{0}_Q, t}, 
    \]
    as $(\mu_{P*Q})_{st} = 0$ whenever $s< \widehat{1}_P$ and $t>\widehat{0}_Q$. To conclude, we may apply the formula in equation~\eqref{eq:recurrence-chow} with $\kappa = \chi$. Unravelling the convolution in that equation, we obtain: 
    \begin{align*}
        \H_{P*Q}(x) &= \sum_{\widehat{0}_P<w < \widehat{1}_P}\overline{\chi}_{\widehat{0}_Pw}(x) \H_{[w,\widehat{1}_P]*Q}(x) + \sum_{\widehat{0}_Q \leq w \leq \widehat{1}_Q} \overline{\chi}_{P*[\widehat{0}_Q,w]}(x) \H_{w,\widehat{1}_Q}(x) \\
        &= \sum_{\widehat{0}_P<w < \widehat{1}_P}\overline{\chi}_{\widehat{0}_Pw}(x) \H_{w,\widehat{1}_P}(x) G_Q(x) + \sum_{\widehat{0}_Q \leq w \leq \widehat{1}_Q} \overline{\chi}_{P}(x) \zeta^{\rev}_{\widehat{0}_Q,w} \H_{w,\widehat{1}_Q}(x) \\
        &= \left(\H_P(x) - \overline{\chi}_P(x)\right)G_Q(x) + \overline{\chi}_P(x)G_Q(x) \\
        &= \H_P(x)G_Q(x).\qedhere
    \end{align*}
\end{proof}

If we apply the last proposition in the case of $P = C_2$ (a chain on two elements), we can deduce a formula for the left augmented Chow polynomial associated to the characteristic function.

\begin{corollary}\label{coro:aug-from-tail}
    Let $P$ be a graded bounded poset. The left $\chi$-augmented Chow polynomial of $P$ is the $\chi$-Chow polynomial of $\operatorname{aug}(P)$, i.e.,
        \[ G_P(x) = \H_{\aug(P)}(x).\]
\end{corollary}

A consequence of the above corollary is that for $\kappa = \chi$, all left augmented Chow polynomials are themselves Chow polynomials. A natural guess for the right $\chi$-augmented polynomial of $P$ would be that it results from augmenting \emph{from the top}, i.e., adding a maximum element above $P$. However, it is easy to see that this does not work.

A further corollary of Proposition~\ref{prop:chow-of-join} is the following product formula for the left augmented Chow polynomial of a join of posets.

\begin{corollary}\label{coro:product-formula-left-aug-chow}
    Let $P$ and $Q$ be two graded bounded posets. The left $\chi$-augmented Chow function of $P*Q$ can be calculated as
    \[ G_{P*Q}(x) = G_P(x)\cdot G_{Q}(x).\]
\end{corollary}

\begin{proof}
    This follows by applying twice the formula in Proposition~\ref{prop:chow-of-join} to the poset $C_2 * P * Q$. This yields
    \[ G_{P*Q}(x) = \H_{C_2 *P}(x) G_{Q}(x) = G_{P}(x) G_{Q}(x). \qedhere\]
\end{proof}

To the best of our knowledge, there is no nice product formula for the right $\chi$-Chow polynomial of a join of posets.

\begin{example}
    The poset $P$ in Example~\ref{ex:some-chow-polys} can be obtained as $C_2 * P' * C_2$, where $P'$ is depicted in Figure~\ref{fig:non-augmented}. In particular, by applying the last proposition twice, it follows that
        \[ \H_{P}(x) = \H_{C_2}(x)\cdot G_{P'}(x)\cdot G_{C_2}(x). \]
    As mentioned earlier, we have $\H_{C_2}(x) = 1$, and it is easy to see that $G_{C_2}(x) = x + 1$, from which we conclude that
    \[ G_{P'}(x) = \frac{1}{x+1} \H_P(x) = x^4+7x^3+13x^2+7x+1.\]
    On the other hand, the product formula of Corollary~\ref{coro:product-formula-left-aug-chow} gives
    \[ G_P(x) = G_{C_2}(x) G_{P'}(x) G_{C_2}(x) = (x+1)(x^4+7x^3+13x^2+7x+1)(x+1).\]
    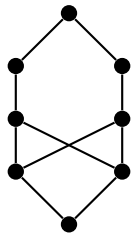
\begin{figure}[ht]
    \centering
	\begin{tikzpicture}  
	[scale=0.7,auto=center,every node/.style={circle,scale=0.8, fill=black, inner sep=2.7pt}] 
	\tikzstyle{edges} = [thick];
	
	\node[] (a2) at (0,1)  {};  
	\node[] (a3) at (-1,2) {};
	\node[] (a4) at (1,2) {};
	\node[] (a5) at (-1,3)  {};  
	\node[] (a6) at (1,3)  {};  
	\node[] (a7) at (-1,4)  {};  
	\node[] (a8) at (1,4)  {};  
	\node[] (a9) at (0,5) {};
	
	%\draw[edges] (a1) -- (a2);  
	\draw[edges] (a2) -- (a3);  
	\draw[edges] (a2) -- (a4);
	\draw[edges] (a3) -- (a5);
        \draw[edges] (a3) -- (a6);
	\draw[edges] (a4) -- (a6);
	\draw[edges] (a4) -- (a5);
	\draw[edges] (a5) -- (a7);
	\draw[edges] (a6) -- (a8);
	\draw[edges] (a7) -- (a9);
	\draw[edges] (a8) -- (a9);
	%\draw[edges] (a9) -- (a10);
	\end{tikzpicture}\caption{The poset $P'$.}\label{fig:non-augmented}
    \end{figure}
\end{example}

\subsection{From characteristic Chow functions to Chow rings of matroids}

We now turn our attention to the case where the poset $P$ is a geometric lattice. As a consequence of Theorem~\ref{thm:chi-chow-graded}, we have the following key connection with Chow rings of matroids. Recall that to any loopless matroid $\M$ one can associate its \emph{Chow ring} via the following procedure. Denote by $\mathcal{L}(\M)$ the lattice of flats of $\M$. Consider the polynomial ring $S = \mathbb{Q}[x_F: F\in \mathcal{L}\smallsetminus\{\varnothing,E\}]$, and the homogeneous ideals 
    \begin{align*}
        I &= \left< x_{F_1} x_{F_2} \,:\, F_1,F_2 \in \mathcal{L}(\M)\smallsetminus\{\varnothing,E\} \text{ are incomparable}\right>,\\
        J &= \left< \sum_{F\ni i} x_F - \sum_{F\ni j} x_F \,:\, i,j\in E\right>.
    \end{align*}

The \emph{Chow ring} of $\M$, denoted $\uCH(\M)$ is defined as the quotient ring $S/(I+J)$. This is a graded ring, admitting a decomposition into $\Q$-vector spaces: $\uCH(\M) = \uCH^0(\M) \oplus \cdots \oplus \uCH^{r-1}(\M)$, where $r = \rk(\M)$ is the rank of the matroid. The following was one of the main results of \cite{ferroni-matherne-stevens-vecchi}. For the sake of completeness, we reprove it in the framework of the present paper.

\begin{theorem}\label{thm:chow-matroid}
    Let $\M$ be a loopless matroid. The $\chi$-Chow polynomial of $\mathcal{L}(\M)$ equals the Hilbert series of the Chow ring $\uCH(\M)$.
\end{theorem}

\begin{proof}
    In \cite[Theorem~1]{feichtner-yuzvinsky} Feichtner and Yuzvinsky computed a Gr\"obner basis for the ring $\uCH(\M)$. From their computation, it follows that $\uCH(\M)$ is isomorphic (as a $\Q$-vector space) to the span of the monomials
        \[ \left\{x_{F_1}^{e_1}\cdots x_{F_m}^{e_m} \enspace : \enspace \varnothing = F_0\subsetneq\cdots \subsetneq F_m :\enspace 0\leq  e_i < \rk(F_i)-\rk(F_{i-1}) - 1 \text{ for $1\leq i\leq m$}\right\}.\]
    From this, it follows that
    \[ \Hilb(\uCH(\M), x) = \sum_{\varnothing = F_0 \subsetneq F_1 \subsetneq \cdots \subsetneq F_m} \prod_{i=1}^m \frac{x ( 1 - x^{\rk(F_i)-\rk(F_{i-1})-1})}{1-x},\]
    which agrees with the $\chi$-Chow polynomial of $\mathcal{L}(\M)$.
\end{proof}

\begin{remark}
    Pagaria and Pezzoli found a Gr\"obner basis for the Chow ring of a loopless polymatroid \cite[Theorem~2.7]{pagaria-pezzoli}. In fact, the preceding proof yields that the Hilbert series of the Chow ring of a loopless polymatroid corresponds to the $\chi$-Chow function of the lattice of flats of a polymatroid (with the caveat that the rank function $\rho$ is the polymatroid rank function). It is tempting to ask if it is possible to construct a ``Chow ring'' for any graded poset $P$. We will address this question in Section~\ref{subsec:chowish-ring}. 
\end{remark}

Let us recapitulate the following construction by Braden, Huh, Matherne, Proudfoot, and Wang \cite{braden-huh-matherne-proudfoot-wang}. The \textit{augmented Chow ring} of a matroid $\M$ is the quotient
        \[
        \CH(\M) = \mathbb{Q}[x_F,\, y_i : F\in\mathcal{L}(\M)\smallsetminus\{E\} \text{ and } i\in E]/{(I+J)},
        \]
    where $E$ is the ground set of the matroid, and the ideals $I$ and $J$ are defined respectively by
    \begin{align*}
        I &= \left< y_i - \sum_{F\notni i} x_F \,:\, i\in E\right>,\\
        J &= \left< x_{F_1} x_{F_2} \,:\, F_1,F_2 \in \mathcal{L}(\M)\smallsetminus\{E\} \text{ are incomparable}\right> + \left< y_i x_F :F\in\mathcal{L}(\M)\smallsetminus\{E\},\, i \notin F\right>.
    \end{align*}

The augmented Chow ring of $\M$ is a graded ring, admitting a decomposition into $\Q$-vector spaces: $\CH(\M) = \CH^0(\M) \oplus \cdots \oplus \CH^r(\M)$, where $r = \rk(\M)$. We have the following description of the left augmented Chow polynomial arising from the characteristic function of a geometric lattice.

\begin{theorem}\label{thm:aug-chow-matroid}
    Let $\M$ be a loopless matroid. The left $\chi$-augmented Chow polynomial equals the Hilbert series of the augmented Chow ring of $\M$.
\end{theorem}

\begin{proof}
    The proof is very similar to that of Theorem~\ref{thm:chow-matroid}. The key ingredient is again the Gr\"obner basis computation of Feichtner and Yuzvinsky~\cite{feichtner-yuzvinsky}, together with \cite[Section~5]{stellahedral}. We omit the details here and refer to~\cite[Proposition~6.3]{ferroni-matherne-stevens-vecchi} instead.
\end{proof}

Recently, Larson~\cite{larson} studied matroid Chow rings from the perspective of Hodge algebras (which are also known as algebras with a \emph{straightening law}). He proved a decomposition of matroid Chow rings and augmented Chow rings in terms of matroid truncations. As a consequence, in \cite[Corollary~3.5]{larson} he derived recursions for the Chow polynomials. As we will now demonstrate, Larson's recursions continue to hold for arbitrary graded bounded posets. Furthermore, they hint at an analogous formula for the right augmented Chow polynomials that is quite non-obvious when specialized to the case of matroids.

To this end, for a graded bounded poset $P$, define the \emph{truncation} of $P$ to be the subposet consisting of all elements of $P$ except the coatoms of $P$. We will denote this poset by $\trunc(P)$. 

\begin{proposition}[Larson's recursions]\label{prop:larson-recursion}
    Let $P$ be a graded bounded poset. The $\chi$-Chow polynomial of $P$ satisfies
    \[ \H_P(x) = 1 + x \sum_{\substack{t\in P\\ \rho(t) > 1}} \H_{\trunc([\widehat{0},t])}(x) .\]
    Furthermore, the left $\chi$-augmented Chow polynomial of $P$ satisfies
    \[ G_P(x) = 1 + x \sum_{t\neq \widehat{0}} G_{\trunc([\widehat{0},t])}(x).\]
\end{proposition}

\begin{proof}
    We proceed by induction on the rank of $P$. If $\rho(P) = 0$, then the statement trivially holds. If $\rho(P) \geq 1$, then we need to show that the right-hand side of the equation counts the chains of elements as in Theorem \ref{thm:chi-chow-graded}. If the chain is empty then the corresponding monomial is $1$. If the chain is not empty then it has a maximal element, say $p_m=t$. The rank of $t$ must be greater than one; otherwise, the corresponding monomial is equal to zero. Moreover, the second to last element in the chain $p_{m-1}$ has to satisfy $\rho_{p_{m-1}t} > 1$, since otherwise $t$ would not give us a non-zero monomial. This gives an explicit bijection between chains ending in $t$ and chains in the truncation of the interval $[\widehat{0},t]$. The formula for $G$ follows immediately, after considering the augmentation of $P$ and using Corollary~\ref{coro:aug-from-tail}.
\end{proof}

The preceding two formulas by Larson motivated us to search for an analogous recursion for the right augmented Chow polynomial. The following result achieves this, but in a much less straightforward way. Furthermore, this is a non-trivial decomposition when viewed under the lens of the Hodge theory of matroids (see the discussion at the end of Section~\ref{sec:unimodality-gamma-flag-char}).

\begin{proposition}\label{prop:right-larson}
    Let $P$ be a graded bounded poset. Then the right augmented $\chi$-Chow polynomial of $P$ satisfies
    \[ F_P(x) = Z_P(x) + x \sum_{\substack{t\in P\\ \rho(t)>1}} \H_{\trunc([\widehat{0},t])}(x) Z_{t\widehat{1}}(x).\]
\end{proposition}

\begin{proof}
    Let $\H$ be the $\chi$-Chow function of a finite graded bounded poset $P$. Let $\trunc(P)$ be the truncation as in Proposition \ref{prop:larson-recursion}. We claim that
    \[
    \sum_{s\leq w \leq t}\H_{sw}(x) \mu_{wt} = \begin{cases}
        1 & \text{if $\rho_{st} = 0$}, \\
        0 & \text{if $\rho_{st} = 1$}, \\
        x\H_{\trunc([s,t])}(x) & \text{otherwise.}
    \end{cases}
    \]
    To see this, let us proceed by induction. If $\rho_{st} \leq 1$, the result is trivial. If $s<t$, then
    \begin{align*}
        \sum_{s\leq w \leq t}\H_{sw}(x) \mu_{wt} &= \mu_{st} + \sum_{s< w \leq t}\H_{sw}(x) \mu_{wt}\\
        &= \mu_{st} + \sum_{s< w \leq t}\left(\sum_{s<u \leq w}\overline{\chi}_{su}(x)\H_{uw}(x) \right) \mu_{wt}\\
        &= \mu_{st} + \sum_{s < u \leq t}\overline{\chi}_{su}(x) \left( \sum_{u \leq w\leq t} \H_{uw}(x) \mu_{wt} \right) \\
        &= \overline{\chi}_{st}(x) + \mu_{st} + x\sum_{\substack{s< u < t \\ \rho_{ut} > 1}}\overline{\chi}_{su}\H_{\trunc([u,1])}(x) \\
        &= x\overline{\chi}_{\trunc([s,t])}(x) + x\sum_{\substack{u \in \trunc([s,t]) \\ u \neq s}}\overline{\chi}_{su}\H_{\trunc([u,1])}(x),
    \end{align*}
    where in the fourth equality we used the inductive hypothesis (notice how asking for $\rho_{ut} > 1$ coincides with considering all the elements except the coatoms, i.e., truncating the poset) and in the fifth equality we used that $\chi_{\trunc(P)}(x) = \chi_P(x) + (x-1)\mu_P$. We can conclude via the definition of Chow function as $\H = -(\overline{\chi})^{-1}$. Now, to prove the formula of the statement, we first write 
    \[
    F = \H f^{\rev} = \H\chi f = \H\mu\upzeta^{\rev}f = \H\mu Z.
    \]
    We can then use the above convolution to write
    \[
    F_P(x) = Z_P(x) + x\sum_{\substack{t\in P\\ \rho(t)>1}}\H_{\trunc([\widehat{0},t])}(x)Z_{t\widehat{1}}(x),
    \]
    and the proof is complete.
\end{proof}

\subsection{The interplay with matroid Hodge theory}\label{subsec:hodge-theory}

The main result of Adiprasito, Huh, and Katz in \cite{adiprasito-huh-katz} is a remarkable feature of Chow rings of matroids. They satisfy a trio of properties known as the \emph{K\"ahler package}. These properties are, respectively, \emph{Poincar\'e duality} (PD), \emph{the Hard Lefschetz theorem} (HL), and \emph{the Hodge--Riemann bilinear relations} (HR). In the Chow ring there is a distinguished map $\deg_{\M}\colon\uCH^{r-1}(\M)\to \mathbb{Q}$ called the \emph{degree map}, defined by requiring that the product of the variables corresponding to a maximal flags of non-empty flats is sent to $1$.

\begin{theorem}[{\cite[Theorem~1.4 and Theorem~6.19]{adiprasito-huh-katz}}]\label{thm:kahler-package}
    Let $\M$ be a loopless matroid of rank $r$ and let $\ell\in \uCH^1(\M)$. Then the following holds.
    \begin{enumerate}
        \item[{\normalfont (PD)}]\label{it:PD} For every $0\leq j\leq r-1$, the bilinear pairing $\uCH^j(\M) \times \uCH^{r-1-j}(\M)\to \mathbb{Q}$ defined by 
            \[ (\upeta,\upxi) \longmapsto \deg_{\M} \left(\upeta\,\upxi\right)\]
        is non-degenerate, i.e., the map $\uCH^j(\M)\to \operatorname{Hom}(\uCH^{r-1-j}(\M),\mathbb{Q})$ defined by \[\upeta \mapsto \left(\upxi \mapsto \deg_{\M}(\upeta\,\upxi)\right)\] is an isomorphism.
        \item[{\normalfont (HL)}]\label{it:HL} For every $0\leq j \leq \left\lfloor\frac{r-1}{2}\right\rfloor$, the map $\uCH^j(\M) \to \uCH^{r-1-j}(\M)$ defined by
            \[ \upxi \longmapsto \ell^{r-1-2j}\, \upxi\]
        is an isomorphism.
        \item[{\normalfont (HR)}]\label{it:HR} For every $0\leq j \leq \left\lfloor\frac{r-1}{2}\right\rfloor$, the symmetric bilinear form $\uCH^j(\M)\times \uCH^j(\M) \to \mathbb{Q}$ defined by
            \[ (\upeta, \upxi) \longmapsto (-1)^j \deg_{\M} \left(\upeta\, \ell^{r-1-2j} \,\upxi\right) \]
        is positive definite when restricted to $\{\alpha \in \uCH^j(\M) : \ell^{r-2j} \alpha = 0\}$. 
    \end{enumerate}
\end{theorem}

Poincar\'e duality guarantees that $\dim \uCH^j(\M) = \dim \uCH^{r-1-j}(\M)$, which in turn says that the Hilbert series of $\uCH(\M)$ is a symmetric polynomial with center of symmetry $\frac{1}{2}(r-1)$. Furthermore, the Hard Lefschetz theorem guarantees that this Hilbert series is unimodal. An analogous result was proved for the augmented Chow ring $\CH(\M)$ by Braden, Huh, Matherne, Proudfoot, and Wang~\cite{braden-huh-matherne-proudfoot-wang}, which thus says that the Hilbert series of the augmented Chow ring is symmetric with center of symmetry $\frac{1}{2} r$ and is unimodal.

A prominent object in the \emph{singular} Hodge theory of matroids is the intersection cohomology module of a matroid. We briefly indicate how it is defined. First, consider the \emph{graded M\"obius algebra} of $\mathcal{L}(\M)$. This has a variable $y_F$ for each element $F\in \mathcal{L}(\M)$ and the product is defined by $y_F\cdot y_{F'} = y_{F\vee F'}$ whenever $\rk(F) + \rk(F') = \rk(F\vee F')$, and zero otherwise. The augmented Chow ring $\CH(\M)$ is naturally an $\H(\M)$-module, and by the Krull-Schmidt theorem there exists a unique (up to isomorphism) indecomposable graded $\H(\M)$-submodule of $\CH(\M)$ containing the degree zero piece $\CH^0(\M)$. This is the  \emph{intersection cohomology module} of $\M$, and it is denoted $\IH(\M)$. The tensor product $\IH(\M)\otimes_{\H(\M)} \mathbb{Q}$ is called the \emph{stalk of $\IH(\M)$ at the empty flat}, and is denoted $\IH(\M)_{\varnothing}$.

One of the main results of Braden, Huh, Matherne, Proudfoot, and Wang~\cite[Theorem~1.9]{braden-huh-matherne-proudfoot-wang} is that the $Z$-polynomial arising from $P=\mathcal{L}(\M)$ using $\chi$ as the $P$-kernel is the Hilbert series of $\IH(\M)$, whereas the right KLS polynomial is the Hilbert series of $\IH(\M)_{\varnothing}$. 

A crucial step in the big induction performed in \cite{braden-huh-matherne-proudfoot-wang} are the so-called \emph{canonical decompositions} appearing in \cite[Definition~3.8]{braden-huh-matherne-proudfoot-wang}. 
    \begin{itemize}%[leftmargin=7\parindent]
        \item $\underline{\operatorname{CD}}(\M)$: The Chow ring can be decomposed as 
        \begin{equation}\label{eq:canonical-dec-chow-ring}
        \uCH(\M) = \underline{\IH}(\M) \oplus \bigoplus_{\varnothing \neq F < E} \underline{\mathrm{K}}_F(\M).
        \end{equation}
        \item $\operatorname{CD}(\M)$: The augmented Chow ring can be decomposed as 
        \begin{equation}\label{eq:canonical-dec-aug-chow}
        \CH(\M) = \IH(\M) \oplus \;\bigoplus_{F < E} \mathrm{K}_F(\M).
        \end{equation}
    \end{itemize}
The modules $\underline{\IH}(\M)$, $\mathrm{K}_F(\M)$, and $\underline{\mathrm{K}}_F(\M)$ have more complicated definitions, so we refer the reader to \cite[Definition~3.1]{braden-huh-matherne-proudfoot-wang}. The numerical canonical decomposition that we proved in equation~\eqref{eq:ncd-f} is precisely what results from computing the graded dimensions of each of the individual summands appearing in the canonical decomposition for the Chow ring \eqref{eq:canonical-dec-chow-ring}. Analogously, the augmented numerical canonical decomposition appearing in equation~\eqref{eq:ncd-G} is what results from ~\eqref{eq:canonical-dec-aug-chow} after computing the Hilbert series. We stress once more the relevance of Theorem~\ref{thm:chow-from-kl} and Theorem~\ref{thm:aug-ncd}, as they are statements that hold regardless of the existence of analogs of the modules.

\begin{remark}
    As said above, the canonical decompositions of Braden, Huh, Matherne, Proufoot, and Wang can be seen as categorical versions of equations~\eqref{eq:ncd-f} and \eqref{eq:ncd-G}. We do not know if it is possible to construct modules that explain the validity of the numerical canonical decompositions appearing in~\eqref{eq:ncd-g} and \eqref{eq:ncd-F}. For example, while by Theorem~\ref{thm:aug-chow-matroid} we know that $G_{\mathcal{L}(\M)}(x)$ is the Hilbert series of the augmented Chow ring of $\M$, we were not able to find in the literature any known structure (e.g., a graded ring) having $F_{\mathcal{L}(\M)}(x)$ as its Hilbert series. 
\end{remark}

\begin{question}
    Let $\M$ be a loopless matroid. Does there exist a graded ring (or module) having $F_{\mathcal{L}(\M)}(x)$ as its Hilbert series? 
\end{question}

Botong Wang (private communication) observed that the apparent geometric object one should consider (in the realizable case) is the closure of the affine cone of the reciprocal plane in the stellahedral variety. In the language of \cite{braden-huh-matherne-proudfoot-wang}, one should be able to provide a module-theoretic definition of this ``right augmented Chow module'', motivated from the geometric picture.

A natural follow-up question is whether this purported ``right augmented Chow module'' satisfies the K\"ahler package. As we will explain below, there are combinatorial reasons to believe so, at least for what concerns the Hard Lefschetz property. Specifically, as we will show in Corollary~\ref{coro:right-aug-gamma-pos}, the right augmented Chow polynomial of a geometric lattice is unimodal (in fact, $\gamma$-positive).

Before finishing this section, we comment that the numerical canonical decomposition appearing in equation~\eqref{eq:ncd-f} gives as an immediate corollary the following new identity at the level of matroids. For the sake of future reference, we will state it using the well-established notation in matroid theory, i.e., $P_{\M}(x)$ will denote the Kazhdan--Lusztig polynomial of $\M$ and $\uH_{\M}(x)$ the Chow polynomial of $\M$.

\begin{corollary}
    Let $\M$ be a loopless matroid of rank $r$ on $E$. Then,
    \[ \uH_{\M}(x) = \frac{x^{\rk(\M)} P_{\M}(x^{-1}) - P_{\M}(x)}{x-1} + \sum_{\substack{F\neq\varnothing\\F\neq E}} \uH_{\M|_F}(x)\, \frac{x^{r-\rk(F)}P_{\M/F}(x^{-1}) - xP_{\M/F}(x)}{x-1}.\]
\end{corollary}

\subsection{Unimodality, gamma-positivity, and flag \texorpdfstring{$h$}{h}-vectors}\label{sec:unimodality-gamma-flag-char}

We now turn back to the more general case of graded bounded posets. Since the left KLS function associated to $\chi$ is non-negative, we can apply Theorem~\ref{thm:kls-positive-chow-unimodal} to conclude the following.

\begin{theorem}\label{thm:chi-chow-unimodal}
    Let $P$ be a graded bounded poset. The $\chi$-Chow polynomial of $P$ is non-negative and unimodal.
\end{theorem}

As explained in the last subsection, in the case of geometric lattices (that is, lattices of flats of matroids), the above result follows from applying the Hard Lefschetz theorem on the Chow ring. In our case, there is a priori no such ring (cf. Section~\ref{subsec:chowish-ring}), but nonetheless the numerical shadow of its validity continues to hold.

When $P=\mathcal{L}(\M)$ is the lattice of flats of the matroid $\M$, then the $\chi$-Chow polynomial and the left $\chi$-Chow polynomial are in fact $\gamma$-positive (see \cite[Theorem~3.25]{ferroni-matherne-stevens-vecchi}). The key ingredient in the proof is the semi-small decomposition for Chow rings and augmented Chow rings proven by Braden, Huh, Matherne, Proudfoot, and Wang \cite{semismall}. 

Recently, using notions in a preliminary version of the theory developed in the present manuscript, Stump \cite{stump} proved a more general result: if $P$ is a poset admitting an $R$-labelling, then the $\chi$-Chow polynomial and the left augmented $\chi$-Chow polynomial are $\gamma$-positive. There is a strict chain of implications:

\[ \text{Geometric lattice} \Longrightarrow \text{Cohen--Macaulay poset} \Longrightarrow \text{Graded poset}.\]

(There are numerous notions of shellability that can be added to the above chain of implications, but we will not deal with them here, so we omit them.) We already know that only assuming that the poset is graded, the Chow function is unimodal, but the following example shows that $\gamma$-positivity may fail.

\begin{example}\label{example:non-gamma-positive}
    Consider the poset $P$ whose Hasse diagram is depicted on the left in Figure~\ref{fig:posets2}. The $\chi$-Chow polynomial equals
    \[ \H_P(x) = x^4 + 7 x^3 + 11 x^2 + 7 x + 1.\]
    This polynomial is not $\gamma$-positive because $\gamma_P(x) = -x^2+3x+1$. Of course, one expects that $P$ is not Cohen--Macaulay, which can be seen from the fact that $\beta_P(\{2,4\}) = -1$.
    \begin{figure}[ht]
    \centering
	\begin{tikzpicture}  
	[scale=0.6,auto=center,every node/.style={circle,scale=0.8, fill=black, inner sep=2.7pt}] 
	\tikzstyle{edges} = [thick];
	
	\node[] (a1) at (0,0) {};  
	\node[] (a2) at (0,1)  {};  
	\node[] (a3) at (-1,2) {};
	\node[] (a4) at (1,2) {};
	\node[] (a5) at (-1,3)  {};  
	\node[] (a6) at (1,3)  {};  
	\node[] (a7) at (-1,4)  {};  
	\node[] (a8) at (1,4)  {};  
	\node[] (a9) at (0,5) {};
	
	\draw[edges] (a1) -- (a2);  
	\draw[edges] (a2) -- (a3);  
	\draw[edges] (a2) -- (a4);  
	\draw[edges] (a3) -- (a5);
	\draw[edges] (a4) -- (a6);
	\draw[edges] (a5) -- (a7);
	\draw[edges] (a6) -- (a8);
	\draw[edges] (a7) -- (a9);
	\draw[edges] (a8) -- (a9);
	\end{tikzpicture}\qquad\qquad
 \begin{tikzpicture}  
	[scale=0.6,auto=center,every node/.style={circle,scale=0.8, fill=black, inner sep=2.7pt}] 
	\tikzstyle{edges} = [thick];
	
	\node[] (a1) at (0,0) {};  
	\node[] (a2) at (-2,1)  {};  
	\node[] (a3) at (-1,1) {};
	\node[] (a4) at (1,1) {};
	\node[] (a5) at (2,1)  {};  
	
        \node[] (a6) at (-3.5,3)  {};  
	\node[] (a7) at (-2.5,3)  {};  
	\node[] (a8) at (-1.5,3)  {};  
	\node[] (a9) at (-0.5,3) {};
	\node[] (a10) at (0.5,3)  {};  
	\node[] (a11) at (1.5,3)  {};  
	\node[] (a12) at (2.5,3)  {};  
	\node[] (a13) at (3.5,3) {};
 
	\node[] (a14) at (0,4) {};
	\node[] (a15) at (0,5) {};
	
        \draw[edges] (a1) -- (a2);
        \draw[edges] (a1) -- (a3);
        \draw[edges] (a1) -- (a4);
        \draw[edges] (a1) -- (a5);
        \draw[edges] (a2) -- (a6);
        \draw[edges] (a2) -- (a7);
        \draw[edges] (a2) -- (a8);
        \draw[edges] (a2) -- (a9);
        \draw[edges] (a2) -- (a10);
        \draw[edges] (a2) -- (a11);
        \draw[edges] (a2) -- (a12);
        \draw[edges] (a2) -- (a13);
        \draw[edges] (a3) -- (a6);
        \draw[edges] (a3) -- (a7);
        \draw[edges] (a3) -- (a8);
        \draw[edges] (a3) -- (a9);
        \draw[edges] (a3) -- (a10);
        \draw[edges] (a3) -- (a11);
        \draw[edges] (a3) -- (a12);
        \draw[edges] (a3) -- (a13);
        \draw[edges] (a4) -- (a6);
        \draw[edges] (a4) -- (a7);
        \draw[edges] (a4) -- (a8);
        \draw[edges] (a4) -- (a9);
        \draw[edges] (a4) -- (a10);
        \draw[edges] (a4) -- (a11);
        \draw[edges] (a4) -- (a12);
        \draw[edges] (a4) -- (a13);
        \draw[edges] (a5) -- (a6);
        \draw[edges] (a5) -- (a7);
        \draw[edges] (a5) -- (a8);
        \draw[edges] (a5) -- (a9);
        \draw[edges] (a5) -- (a10);
        \draw[edges] (a5) -- (a11);
        \draw[edges] (a5) -- (a12);
        \draw[edges] (a5) -- (a13);
        \draw[edges] (a14) -- (a6);
        \draw[edges] (a14) -- (a7);
        \draw[edges] (a14) -- (a8);
        \draw[edges] (a14) -- (a9);
        \draw[edges] (a14) -- (a10);
        \draw[edges] (a14) -- (a11);
        \draw[edges] (a14) -- (a12);
        \draw[edges] (a14) -- (a13);
        \draw[edges] (a14) -- (a15);
	\end{tikzpicture}\caption{The posets of Example~\ref{example:non-gamma-positive} and Remark~\ref{rem:F-horrible}.}\label{fig:posets2}
 \end{figure}
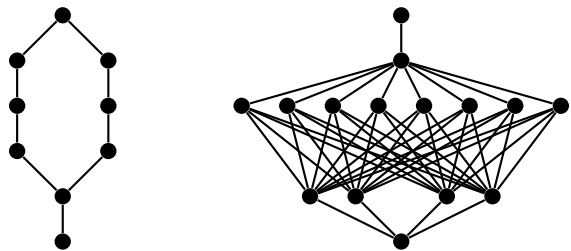
\end{example}

\begin{theorem}\label{thm:cohen-macaulay-chow-gamma-positive}
    Let $P$ be a Cohen--Macaulay poset. The $\chi$-Chow polynomial of $P$ is $\gamma$-positive. 
\end{theorem}

The proof that we will provide is inherently technical, as one cannot a priori rely on any nice labelling for the poset. We will in fact prove a general relation between the flag $h$-vector and the $\chi$-Chow polynomial of $P$ (cf. Theorem~\ref{thm:chow-from-flag-h}), from which Theorem~\ref{thm:cohen-macaulay-chow-gamma-positive} and Stump's result will be immediate corollaries.

Before moving to the proof of Theorem~\ref{thm:cohen-macaulay-chow-gamma-positive}, let us point out that this result along with Corollary~\ref{coro:aug-from-tail} imply that if $P$ is Cohen--Macaulay, then the left augmented $\chi$-Chow polynomial $G_P(x)$ is $\gamma$-positive. The reason is that $\aug(P)$ is Cohen--Macaulay too, as the order complex $\Delta(\aug(P))$ is a cone over $\Delta(P)$ and coning preserves Cohen--Macaulayness.

\begin{corollary}
    Let $P$ be a Cohen--Macaulay poset. The left augmented $\chi$-Chow polynomial $G_P(x)$ is $\gamma$-positive.
\end{corollary}

We will prepare the proof with a few preliminary lemmas. The following notation will be instrumental throughout the proof. 

\vspace{-4mm}
\[ \text{$S\subseteq \mathbb{Z}$ is said to be \emph{good} if $S$ does not contain $1$ nor any two consecutive integers}.\]

We now define an auxiliary family of polynomials. For each non-negative integer $r$, the polynomial $1 + x + \cdots + x^r$ is symmetric, so we can consider its associated $\gamma$-polynomial, which we will denote by $W_r(x)$. We additionally define $W_{-1}(x) = 0$. Notice that $W_r(x)$ has negative coefficients for $r \geq 2$, as the polynomial $1 + \cdots + x^r$ is not $\gamma$-positive when $r>1$.

\begin{lemma}
    The polynomial $W_r(x)$ can be computed as
    \[
        W_r(x) = \sum_{j=0}^{\lfloor r/2\rfloor} (-1)^j \binom{r-j}{j}\, x^j = \sum_{\substack{T\subseteq [m]\\ T \text{ good}}} (-x)^{|T|}.
    \]
\end{lemma}

\begin{proof}
    The first equality follows from an elementary computation using the definition of the polynomial $W_r(x)$ as the $\gamma$-polynomial of $1 + \cdots + x^r$. To prove the second equality, it suffices to show that
    \[
    \left|\left\{ T \subseteq [r] : |T| = j, \text{ $T$ good} \right\}\right| = \binom{r-j}{j}.
    \]
    We do this by induction. For $j = 0$, the last equality is trivially true. Now, for $j\geq 1$, we enumerate the good sets as follows. We choose an element $i \in [2,m]$ and count how many good sets of size $j$ have $i$ as minimum element. Note that $i \leq m - 2(j-1)$, otherwise we cannot find  $j-1$ elements bigger than $i$ to form a good set, as necessarily there would be two consecutive integers among the choice. Inductively, we shall assume that the number of good sets of size $j-1$ in the interval $[i,m]$ is $\binom{m-i-j}{j-1}$. Notice that
    \[
    \sum_{i=2}^{m-2j+2}\binom{m-i-j}{j-1} = \sum_{\ell=0}^{m-2j}\binom{j-1+\ell}{j-1}= \binom{m-j}{j},
    \]
    where the last equality follows from applying the ``hockey-stick identity''.
\end{proof}

We will also need the following refinement of the previous lemma.

\begin{lemma}\label{lemma:good-containing}
    Let $T = \{r_1,\ldots,r_m\}\subseteq [r-1]$ be a good set. Then, we have
    \[
    \left[\prod_{i=1}^m W_{r_i - r_{i-1} - 2}(x) \right]W_{r-r_m-1}(x) = \sum_{\substack{S\supseteq T \\ S \text{ good}}}(-x)^{|S\smallsetminus T|}.
    \]
\end{lemma}

\begin{proof}
    To show this, we fix a good set $T$ and partition the interval $[r]$ as
    \[
    [r] = [1,r_1]\sqcup [r_1+1,r_2]\sqcup \cdots \sqcup [r_{m-1}+1,r_m]\sqcup [r_m+1,r].
    \]
    Notice that by taking the last element of each block (except for the last one) we recover $T$. Moreover, each interval $T_i = [r_{i-1}+1,r_i]$ has length $r_i-r_{i-1}$ (where we set $r_0 = 0$ and $r_{m+1}=r$). The good sets $S\supseteq T$ arise by picking subsets $S_i \subseteq T_i$ such that
    \begin{itemize}
        \item $r_i \in S_i$,
        \item $r_{i-1}+1 \notin S_i$, and
        \item $S_i$ does not contain two consecutive integers.
    \end{itemize}
    These are clearly in bijection with the good sets of $[r_{i-1}+1, r_i - 1]$, which are counted by $W_{r_i - r_{i-1}-2}(x)$ for $i \leq s$ and $W_{r-r_m-1}(x)$ for $i=m+1$. We conclude using the previous lemma.
\end{proof}
The relevance of the polynomials $W_r(x)$ stems from the following interpretation of the numerical canonical decomposition for $\chi$-Chow polynomials. To avoid overloading our notation, we will write
    \[ \gamma_P(x) := \gamma(\H_P, x).\]

\begin{lemma}\label{lemma:gamma-recursion}
    Let $P$ be a graded bounded poset of rank $r$. The $\chi$-Chow polynomial of $P$ satisfies the following recursion:
    \[ \gamma_P(x) = W_{r-1}(x) + x \,\sum_{\widehat{0} < t < \widehat{1}} W_{\rho(t)-2}(x)\, \gamma_{[t,\widehat{1}]}(x).\]
    In particular, the polynomial $\gamma_P(x)$ can be computed with the following non-recursive formula:
    \[
    \gamma_P(x) = \sum_{\substack{S\subseteq [r-1]\\S \text{ good}\\ S=\{r_1,\ldots,r_m\}}}x^{m}\alpha_P(S) \left[\prod_{i=1}^m W_{r_i - r_{i-1} - 2}(x) \right]W_{r-r_s-1}(x).
    \]
\end{lemma}

\begin{proof}
    From the numerical canonical decomposition \eqref{eq:ncd-g}, the fact that $g$ is identically one allows us to write
    \[ \H_P(x) = \frac{x^r - 1}{x-1} + \sum_{\widehat{0} < w < \widehat{1}} \frac{x^{\rho(w)} - x}{x-1} \,\H_{[w,\widehat{1}]}(x).\]
    All the summands in the above display have the same center of symmetry, and hence the $\gamma$-polynomial can be computed using the asserted recursion. Now, by composing this recursion with itself, we can rewrite $\gamma_P(x)$ as a sum over chains of $P$:
    \[ \gamma_P(x) = \sum_{\widehat{0} = t_0 < \cdots < t_m \leq \widehat{1}} W_{r-\rho(t_m)-1} \prod_{i=1}^m x\,W_{\rho(t_i)-\rho(t_{i-1})-2}. \]
    Since by definition $W_{-1}(x)=0$, we shall assume that each summand is indexed by a chain whose elements have ranks forming a good set. Thus, we obtain the non-recursive formula of our statement.
\end{proof}

By combining all the previous lemmas, we may prove the central statement of this section.

\begin{theorem}\label{thm:chow-from-flag-h}
    Let $P$ be a graded bounded poset of rank $r$. The $\chi$-Chow polynomial of $P$ can be computed from the flag $h$-vector as follows:
    \begin{equation}\label{eq:gamma-flag}
        \gamma_P(x) = \sum_{\substack{S\subseteq [r-1]\\ S \text{ good}}} \beta_P(S)\, x^{|S|}.
    \end{equation}
\end{theorem}

\begin{proof}
    Define the polynomial $\widetilde{\gamma}_P(x)$ as the right-hand side in equation~\eqref{eq:gamma-flag}, and rewrite $\widetilde{\gamma}_P(x)$ in terms of the flag $f$-vector. One obtains
    \begin{align*}
        \widetilde{\gamma}_P(x) &= \sum_{\substack{S\subseteq [r-1]\\ S \text{ good}}} \left(\sum_{T\subseteq S} (-1)^{|S|-|T|} \alpha_P (T)\right)\, x^{|S|}\\
        &= \sum_{\substack{T\subseteq [r-1]\\ T \text{ good}}} (-1)^{|T|} \alpha_P(T) \left(\sum_{\substack{S\supseteq T\\ S \text{ good}}} (-x)^{|S|}\right).
    \end{align*}
    Notice that we interchanged the order of the summation, and used that if $S$ is good, then any subset $T\subseteq S$ is good too. We can apply Lemma~\ref{lemma:good-containing}, to obtain
    \[ \widetilde{\gamma}_P(x) = \sum_{\substack{T\subseteq [r-1]\\ T \text{ good}}} (-1)^{|T|} \alpha_P(T)(-x)^{|T|} \left[\prod_{i=1}^m W_{r_i - r_{i-1} - 2}(x) \right]W_{r-r_s-1}(x),\]
    and now the formula in Lemma~\ref{lemma:gamma-recursion} says that the right-hand side is $\gamma_P(x)$.
\end{proof}

As a consequence of the above statement and the non-negativity of the flag $h$-vector of a Cohen--Macaulay poset, we conclude the validity of Theorem~\ref{thm:cohen-macaulay-chow-gamma-positive}. Observe that this also proves Stump's result, because posets with an $R$-labelling have non-negative flag $h$-vectors. However, the classes of Cohen--Macaulay posets and posets with $R$-labellings do not compare to each other.

In an independent work, and simultaneously to the writing of the present paper, Liao \cite{liao2} proved a conjecture of Angarone, Nathanson, and Reiner \cite{angarone-nathanson-reiner} on the equivariant $\gamma$-positivity of Chow rings of matroids, using techniques that are similar to our proof of Theorem~\ref{thm:cohen-macaulay-chow-gamma-positive}. 

It is sensible to ask if the stronger property of real-rootedness holds. When $P$ is a geometric lattice, this is an outstanding conjecture by Ferroni and Schr\"oter \cite[Conjecture~8.18]{ferroni-schroter} and by Huh--Stevens \cite[Conjecture~4.3.3]{stevens-bachelor}.

\begin{remark}\label{rem:not-RR-CM}
    Based on numerous experiments, in an earlier version of this manuscript, we conjectured that real-rootedness might hold for Cohen--Macaulay posets. We also pointed out that that conjecture beared a resemblance to the Neggers--Stanley conjecture, for which there are known counterexamples by Stembridge \cite{stembridge} (see also the previous work by Br\"and\'en \cite{branden-neggers}). At that time we were unable to find adequate modifications of those examples that would yield a non-real-rooted Chow polynomial (an obstruction to test a high number of examples is that computing the Chow polynomial of similar posets of rank $17$ on approximately $50$ elements takes a few hours on a standard computer). In May 2026, we found, after modifying Stembridge's example in the right way, a Cohen--Macaulay poset whose Chow polynomial is not real-rooted. The poset is depicted in Figure~\ref{fig:CM-not-RR}: it is a distributive lattice, in particular it is modular, and thus semimodular, and EL-shellable. Some colleagues suggested to restrict the conjecture from Cohen--Macaulay posets to posets with an EL-labelling, or to extend it to posets with a nonnnegative flag $h$-vector; this example shows that neither of those versions is true. We note, however, that our example is far from being a geometric lattice, as atomicity fails in a dramatic way (this poset has only one atom).

    \begin{figure}[ht]
    \begin{tikzpicture}[scale = 0.7,
    x=1.15cm,
    y=0.70cm,
    auto=center,
    every node/.style={circle, fill=black, inner sep=2.5pt},
    edges/.style={thick}
]

%--- nodes ---
\foreach \i/\x/\y in {%
1/0/0,%
2/0/1,%
3/-1.2/2, 4/1.2/2,%
5/-1.2/3, 7/1.2/3,%
6/-2.4/4, 8/0/4, 11/2.4/4,%
9/-2.4/5, 12/0/5, 16/2.4/5,%
10/-3.6/6, 13/-1.2/6, 17/1.2/6, 22/3.6/6,%
14/-2.4/7, 18/0/7, 23/2.4/7,%
15/-3.6/8, 19/-1.2/8, 24/1.2/8, 29/3.6/8,%
20/-2.4/9, 25/0/9, 30/2.4/9,%
21/-3.6/10, 26/-1.2/10, 31/1.2/10, 37/3.6/10,%
27/-2.4/11, 32/0/11, 38/2.4/11,%
28/-3.6/12, 33/-1.2/12, 39/1.2/12, 45/3.6/12,%
34/-2.4/13, 40/0/13, 46/2.4/13,%
35/-2.4/14, 41/0/14, 47/2.4/14,%
36/-2.4/15, 42/0/15, 48/2.4/15,%
43/-1.2/16, 49/1.2/16,%
44/-1.2/17, 50/1.2/17,%
51/0/18%
}{
    \node (n\i) at (\x,\y) {};
}

%--- edges ---
\path[edges]
(n1) edge (n2)

(n2) edge (n3) edge (n4)

(n3) edge (n5)
(n4) edge (n5) edge (n7)

(n5) edge (n6) edge (n8)
(n6) edge (n9)

(n7) edge (n8) edge (n11)
(n8) edge (n9) edge (n12)

(n9) edge (n10) edge (n13)
(n10) edge (n14)

(n11) edge (n12) edge (n16)
(n12) edge (n13) edge (n17)

(n13) edge (n14) edge (n18)
(n14) edge (n15) edge (n19)

(n15) edge (n20)

(n16) edge (n17) edge (n22)
(n17) edge (n18) edge (n23)

(n18) edge (n19) edge (n24)
(n19) edge (n20) edge (n25)

(n20) edge (n21) edge (n26)
(n21) edge (n27)

(n22) edge (n23)
(n23) edge (n24) edge (n29)

(n24) edge (n25) edge (n30)
(n25) edge (n26) edge (n31)

(n26) edge (n27) edge (n32)
(n27) edge (n28) edge (n33)

(n28) edge (n34)

(n29) edge (n30)
(n30) edge (n31) edge (n37)

(n31) edge (n32) edge (n38)
(n32) edge (n33) edge (n39)

(n33) edge (n34) edge (n40)
(n34) edge (n35) edge (n41)

(n35) edge (n36) edge (n42)
(n36) edge (n43)

(n37) edge (n38)
(n38) edge (n39) edge (n45)

(n39) edge (n40) edge (n46)
(n40) edge (n41) edge (n47)

(n41) edge (n42) edge (n48)
(n42) edge (n43) edge (n49)

(n43) edge (n44) edge (n50)
(n44) edge (n51)

(n45) edge (n46)
(n46) edge (n47)
(n47) edge (n48)
(n48) edge (n49)
(n49) edge (n50)
(n50) edge (n51);

\end{tikzpicture}\caption{A distributive lattice whose Chow polynomial is not real-rooted.}\label{fig:CM-not-RR}
\end{figure}
\end{remark}

Before ending this section, we address the case of the right augmented Chow polynomial.

\begin{proposition}
    Let $\mathcal{C}$ be a hereditary class of graded bounded posets that is closed under truncations. Consider the $Z$-function and the right augmented Chow function arising from $\chi$. 
    \begin{enumerate}[\normalfont(i)]
        \item If $Z$ is unimodal on all posets in $\mathcal{C}$, then so is $F$.
        \item If $\H$ and $Z$ are $\gamma$-positive on all posets of $\mathcal{C}$, then so is $F$.
    \end{enumerate}
\end{proposition}

\begin{proof}
    We rely on the right augmented version of Larson's decomposition, proved in Proposition~\ref{prop:right-larson}. The formula proved in that statement shows that
    \[ F_{st}(x) = Z_{st}(x) + x\sum_{w\neq s} \H_{\trunc([s,w])}(x)\, Z_{wt}(x).\]
    Each summand $\H_{\trunc([s,w])}(x)\, Z_{wt}(x)$ is a product of two symmetric polynomials. Since both polynomials are always unimodal (resp. $\gamma$-positive), then so is their product. Furthermore, all of the summands are symmetric with center of symmetry $\frac{1}{2}(\rho(t) - \rho(s) -2)$. The factor $x$ preserves unimodality (resp. $\gamma$-positivity), and shifts the center of symmetry to $\frac{1}{2}(\rho(t)-\rho(s))$. By adding the unimodal (resp. $\gamma$-positive) term $Z_{st}(x)$, which has the correct center of symmetry, the proof is complete.
\end{proof}

One of the main results of \cite[Theorem~4.7]{ferroni-matherne-stevens-vecchi} establishes that when $P$ is a geometric lattice, the $Z$-polynomial arising from the characteristic function is $\gamma$-positive (unimodality was proved first via the Hard Lefschetz theorem, by Braden, Huh, Matherne, Proudfoot, and Wang \cite[Theorem~1.2(2)]{braden-huh-matherne-proudfoot-wang}), proving a conjecture in \cite{ferroni-nasr-vecchi}. In particular, we obtain the following corollary to the above proposition.

\begin{corollary}\label{coro:right-aug-gamma-pos}
    If $P$ is a geometric lattice, then $F_P(x)$ is $\gamma$-positive.
\end{corollary}

\begin{remark}\label{rem:F-horrible}
     It is reasonable to ask whether only assuming the Cohen--Macaulayness of $P$ would be enough to conclude the $\gamma$-positivity of the right augmented Chow polynomial and the $Z$-polynomial. The answer is negative in a strong sense: there exist Cohen--Macaulay posets for which the right augmented Chow polynomial and the $Z$-polynomial fail to even be positive. For example, consider the poset $P$ depicted on the right in Figure~\ref{fig:posets2}. The number of elements of each rank in $P$ is $1$, $4$, $8$, $1$, and $1$, where all the comparability relations between consecutive levels are added. It is straigthforward to check that this is a Cohen--Macaulay poset. We have
     \begin{align*} 
        F_P(x) &= x^4 + 11x^3 - x^2 + 11x + 1,\\
        Z_P(x) &= x^4 + x^3 - 20x^2 + x + 1.
     \end{align*}
     
     Note that the proof of $\gamma$-positivity we have for geometric lattices relies on the heavy machinery of the singular Hodge theory of matroids, because the proof of the $\gamma$-positivity of the $Z$-polynomial uses the non-negativity of the Kazhdan--Lusztig polynomials of matroids \cite[Remark~4.8]{ferroni-matherne-stevens-vecchi}. In this case, the poset $P$ has right KLS polynomial equal to $f_P(x) = -3x+1$. This example shows that there is a huge leap between the behavior of $F$ and $Z$ for Cohen--Macaulay posets versus geometric lattices. For the sake of clarity, we have summarized these and further facts about characteristic Chow functions in Table~\ref{table:chipolyproperties}.
\end{remark}

\begin{table}[htbp]
    \centering

\begin{tabular}{| l | l | l | l | l |} \hline
 & \textbf{positivity} & \textbf{unimodality} & \textbf{$\gamma$-positivity} & \textbf{real-rootedness} \\ \hline
\multirow{3}{.8em}{$\H$} & \multirow{3}{8.5em}{true for all posets} & \multirow{3}{8.5em}{true for all posets} & \multirow{3}{8.5em}{true for all CM posets, false in general} & \multirow{3}{7em}{conjectured for all geometric lattices} \\ {} & {} & {} & {} & {} \\ {} & {} & {} & {} & {} \\ \hline
\multirow{3}{.8em}{$F$} & \multirow{3}{8.5em}{true for geometric lattices, false for CM and general posets} & \multirow{3}{8.5em}{true for geometric lattices, false for CM and general posets} & \multirow{3}{8.5em}{true for geometric lattices, false for CM and general posets} & \multirow{3}{7em}{unknown for geometric lattices} \\ {} & {} & {} & {} & {} \\ {} & {} & {} & {} & {} \\ \hline
\multirow{3}{.8em}{$G$} & \multirow{3}{8.5em}{true for all posets} & \multirow{3}{8.5em}{true for all posets} & \multirow{3}{8.5em}{true for CM posets, false in general} & \multirow{3}{7em}{conjectured for all geometric lattices} \\ {} & {} & {} & {} & {} \\ {} & {} & {} & {} & {} \\ \hline
\end{tabular}
\caption{A list of properties for polynomials associated to the $\chi$-Chow polynomial of various posets.}    \label{table:chipolyproperties}
\end{table}

\subsection{A Chow ring for arbitrary graded posets?}\label{subsec:chowish-ring}

It is hard to resist the temptation of asking if one can associate to each graded bounded poset $P$ a graded Artinian ring $A(P)$ in such a way that the Hilbert series of $A(P)$ matches the Chow polynomial $\H_P(x)$.

Furthermore, if such a ring exists, it is sensible to expect that it satisfies the following properties:
    \begin{enumerate}[(i)]
        \item Poincar\'e duality, because $\H_P(x)$ is palindromic by Proposition~\ref{prop:degree-and-symmetry}.
        \item A version of the Hard Lefschetz theorem, because $\H_P(x)$ is unimodal by Theorem~\ref{thm:chi-chow-unimodal}.
        \item Some analog of Larson's decompositions \cite{larson}, because $\H_P(x)$ satisfies the corresponding recursions by Proposition~\ref{prop:larson-recursion}.
        \item When $P$ is a geometric lattice, $A(P)$ is isomorphic to the Chow ring defined by Feichtner and Yuzvinsky in \cite{feichtner-yuzvinsky}.
    \end{enumerate}

We have attempted to construct such a ring, but we have not been able to do so. To give an example of a reasonable guess, consider the polynomial ring $S = \mathbb{Q}[x_s: s\in P\smallsetminus\{\widehat{0}\}]$, and consider the ideal $I\subseteq S$ generated by the polynomials:
    \begin{align} 
        &x_{s}x_{t} & &\text{for all $s$ and $t$ that do not compare},\\
        &x_s\left(\sum_{w\geq t} x_w\right)^{\rho(t)-\rho(s)}&   &\text{for all $s<t$},\\
        &\left(\sum_{t\geq s} x_t\right)^{\rho(s)} &  &\text{for all $s$}.
    \end{align}
By \cite[Theorem~1]{feichtner-yuzvinsky}, when $P$ is a geometric lattice, the quotient ring $S/I$ is the Chow ring of any matroid having $P$ as its lattice of flats. However, if $P$ is not a geometric lattice, the ring obtained by taking the quotient $S/I$ does not yield the desired polynomial $\H_P(x)$, as its Hilbert series often fails to be palindromic.

Note that some other previously obtained results, such as the formula $\H_{P*Q}(x) = \H_P(x)\cdot \H_{\aug(Q)}$ for the join of two posets, impose additional restrictions on any hypothetical ring that categorifies the $\chi$-Chow polynomial successfully.

\section{Eulerian Chow functions of Eulerian posets}\label{sec:five}

In this section, we will once again deal with finite graded bounded posets. In particular, the discussion at the beginning of Section~\ref{sec:four} applies.

\subsection{Basics of Eulerian Chow polynomials}

Let $n\in \mathbb{Z}_{\geq 0}$. From a $d$-dimensional convex polytope $\mathscr{P}\subseteq \mathbb{R}^n$ one can construct the poset $P$ of all of the faces of $\mathscr{P}$, where the order is given by inclusion. Note that $P$ is bounded and graded, and its rank equals $d$. The poset $P$ fulfills a number of important properties (see, e.g., \cite[Theorem~2.7]{ziegler}). One of the most fundamental is that $P$ is Eulerian.

\begin{definition}
    Let $P$ be a graded bounded poset. We say that $P$ is \emph{Eulerian} if the M\"obius function satisfies $\mu_{st} = (-1)^{\rho(t) - \rho(s)}$ for every $s\leq t$ in $P$.
\end{definition}

We will assume that the reader is acquainted with the basic properties of Eulerian posets, and we refer to \cite[Section~3.16]{stanley-ec1} for a more detailed treatment of this topic. We will often use that $P$ is Eulerian if and only if every interval $[s,t]$ contains the same number of elements of odd rank and even rank. As an example, the poset $Q$ appearing in Figure~\ref{fig:poset1} is Eulerian. It is easy to see that this poset $Q$ cannot be the face poset of a polytope, because face posets of polytopes are atomic lattices while $Q$ is clearly not. 

The following provides a characterization of Eulerian posets in terms of kernels.

\begin{proposition}[{\cite[Proposition~7.1]{stanley-local}}]
    Let $P$ a finite graded bounded poset. Then $P$ is Eulerian if and only if the element $\varepsilon\in \mathcal{I}_{\rho}(P)$ given by $\varepsilon_{st}(x) = (x-1)^{\rho(t) - \rho(s)}$ is a $(P,\rho)$-kernel.
\end{proposition}

We will often refer to $\varepsilon$ as the \emph{Eulerian $P$-kernel}. Correspondingly, if $P$ is an Eulerian poset, the $\varepsilon$-Chow function of $P$ will be called the \emph{Eulerian Chow function} of $P$. 

The left KLS polynomial $g_P(x)$ arising from the $P$-kernel $\varepsilon$ in an Eulerian poset $P$ is often called the \emph{toric $g$-polynomial of $P$}.\footnote{We point out an ambiguity in the literature. In some sources, when $\mathscr{P}$ is a polytope, the \emph{toric $g$-polynomial of $P$} is defined to be the toric $g$-polynomial of the poset of faces of $\mathscr{P}$ ordered under \emph{containment} (as opposed to under inclusion). This will not be too important for us, as we will be mainly focusing on the posets rather than the polytopes.} We refer to the work of Bayer and Ehrenborg \cite{bayer-ehrenborg} for a thorough study of toric $g$-polynomials of Eulerian posets.  As Stanley points out in his book \cite[p.~315]{stanley-ec1}, the toric $g$-polynomial is an exceedingly subtle invariant of the poset $P$. It is not difficult to see that the right KLS polynomial $f_P(x)$ arising from $\varepsilon$ equals the toric $g$-polynomial of the dual poset $P^*$.

\begin{example}\label{example:eulerian}
    Consider the poset $Q$ depicted on the right of Figure~\ref{fig:poset1}.
    Below we include the resulting KLS, Chow, and augmented Chow polynomials of this poset using the Eulerian kernel:
    \begin{align*}
        f_Q(x) &= -6x^2 -x + 1,\\
        g_Q(x) &= -6x^2 -x + 1,\\
        \H_Q(x) &= x^4 + 12x^3 + 6x^2 + 12x + 1,\\
        F_Q(x) &= x^5 + 16x^4 + 18x^3 + 18x^2 + 16x + 1,\\
        G_Q(x) &= x^5 + 16x^4 + 18x^3 + 18x^2 + 16x + 1.
    \end{align*}
    Unlike the case of $\chi$-Chow polynomials, for which we always had the non-negativity of the left KLS function $g$ and therefore the unimodality of $\H_Q(x)$ via Theorem~\ref{thm:kls-positive-chow-unimodal}, for $\varepsilon$-Chow polynomials these phenomena do not persist.
\end{example}

The following is the key result that allows us to describe Eulerian Chow polynomials in a transparent way.

\begin{theorem}\label{thm:chow-eulerian-h-vector-order-complex}
    Let $P$ be an Eulerian poset. The Eulerian Chow polynomial of $P$ equals the $h$-polynomial of the order complex $\Delta(P)$.
\end{theorem}

The preceding result says that the $\varepsilon$-Chow polynomial encodes precisely the number of chains of each size in $P$. As Stanley asserts in \cite[p.~310]{stanley-ec1} the class of Eulerian posets enjoys remarkable properties involving the enumeration of chains, and the above result is one further manifestation of this phenomenon.

For each integer $n\geq 1$, consider the element $\upzeta^n\in \mathcal{I}(P)$ given by
    \[ \upzeta^n := \underbrace{\upzeta \cdots \upzeta}_{\text{$n$ times}}.\]
By \cite[Theorem~3.12.1(c)]{stanley-ec1}, it is known that $\upzeta^n_{st}$ equals the number of \emph{multichains} in the closed interval $[s,t]\subseteq P$ having length $n-1$. That is, $\upzeta^n_{st}$ is the number of ways of choosing elements $t_1,\ldots, t_{n-1}$ such that $s\leq t_1\leq t_2 \leq \cdots \leq t_{n-1}\leq t$, where repetitions are allowed. Stanley calls the map $n\mapsto \upzeta^n_P$ the ``$Z$-polynomial of $P$'', but we will not use this name in hopes of avoiding confusion with the $Z$-function defined in Definition~\ref{def:zeta-functions}. When $P$ is Eulerian, we have $(-1)^{\rho(s)-\rho(t)}\upzeta_{st}^n = \mu_{st}^n$ for all $n\geq 0$ (see \cite[Proposition~3.16.1]{stanley-ec1}).

\begin{proof}[Proof of Theorem~\ref{thm:chow-eulerian-h-vector-order-complex}]
    If $\rho(P) = 0$, then the formula is clearly true, as both polynomials in the statement equal $1$. 
    
    For the induction step, we rely on a technical result about the enumeration of chains in graded bounded posets. From the formula in \cite[Exercise~3.157(ii)]{stanley-ec1}, it is known that
        \[ h(\Delta(P),x) = (1-x)^{\rho(P)+1} \sum_{n=0}^{\infty} \upzeta^n_P \, x^{n+1}.\]
    When $\rho(P) >0$, the recursion for Chow polynomials in equation~\eqref{eq:recurrence-chow} gives
    \[
    \H_P(x) = \sum_{t \neq \widehat{0}}(x-1)^{\rho(t) - 1}\H_{t,\widehat{1}}(x).
    \]
    Every Chow polynomial appearing on the right-hand side of the last equation is on a poset of smaller rank, hence by induction
    \begin{align*}
    \H_P(x) &= \sum_{t \neq \widehat{0}}(x-1)^{\rho(t) - 1} h(\Delta([t,\widehat{1}]))\\
    &= \sum_{t \neq \widehat{0}}(x-1)^{\rho(t) - 1}(1-x)^{\rho(P) - \rho(t) +1} \sum_{n\geq 0}\upzeta^n_{t,\widehat{1}}\,x^n \\
    &= (1-x)^{\rho(P) +1} \sum_{t \neq \widehat{0}}(-1)^{\rho(t)}\frac{1}{1-x} \sum_{n\geq 0}\upzeta^n_{t,\widehat{1}}\,x^n. 
    \end{align*}
    Now, expanding $\frac{1}{1-x} = \sum_{\ell\geq 0} x^\ell$ and reordering the terms, we have
    \begin{align*}
        \H_P(x) &= (1-x)^{\rho(P) +1} \sum_{\ell\geq 0}\sum_{t \neq \widehat{0}}(-1)^{\rho(t)} \sum_{n\geq 0}\upzeta^n_{t,\widehat{1}}\, x^{\ell+n}.
    \end{align*}
    By using the change of variable $m := \ell + n$, we obtain
    \[
    \H_P(x) = (1-x)^{\rho(P) +1}\sum_{m \geq 0}\left( \sum_{n=0}^m \sum_{t \neq \widehat{0}}(-1)^{\rho(t) -1}\upzeta^n_{t,\widehat{1}} \right)x^m. 
    \]
    To conclude the proof of the theorem, it remains to verify that 
    \[
    \upzeta^m_P = \sum_{n=0}^m \sum_{t \neq \widehat{0}}(-1)^{\rho(t) -1}\upzeta^n_{t,\widehat{1}} .
    \]
    Observe that the condition on $P$ being Eulerian implies that $(-1)^{\rho(t)}=\mu_{\widehat{0},t}$. So, by using that $\mu\cdot \upzeta^n = \upzeta^{n-1}$ for $n\geq 1$, we have
    \[ \sum_{n=0}^m \sum_{t \neq \widehat{0}}(-1)^{\rho(t) -1}\upzeta^n_{t,\widehat{1}} = - \sum_{n=0}^m \sum_{t \neq \widehat{0}}\mu_{\widehat{0},t}\upzeta^n_{t,\widehat{1}} = -
    \left( \mu_{P} + \sum_{n=1}^m \sum_{t \neq \widehat{0}}\mu_{\widehat{0},t}\upzeta^n_{t,\widehat{1}}\right)
    =
    - \mu_P - \sum_{n=1}^m (\upzeta^{n-1}_P-\upzeta^n_P),  \]
    and the last sum telescopes and cancels the remaining $\mu_P$. This gives $\upzeta_P^m$, as desired.
\end{proof}

An immediate conclusion from the last result is that the $\varepsilon$-Chow polynomial of an Eulerian poset $P$ is a non-negative combination of entries of the flag $h$-vector of $P$. In particular, if $P$ is Cohen--Macaulay then $\H_P(x)$ has non-negative coefficients. Even though there exist many Eulerian posets that are not Cohen--Macaulay (see, e.g., the poset on the right in Figure~\ref{fig:poset1}), it is considerably harder to construct an Eulerian poset whose flag $h$-vector attains a negative entry. A subtle result by Bayer and Hetyei \cite{bayer-hetyei} shows that in fact all Eulerian posets of rank at most $6$ have a non-negative flag $h$-vector. However, they construct a very complicated example \cite[Figure~2]{bayer-hetyei} in rank $7$ for which the flag $h$-vector attains a negative entry. Furthermore, modulo the Dehn--Sommerville relations, this flag $h$-vector attains a \emph{single} negative entry, equal to $-1$. We refer to \cite[Solution to Exercise 193(b)]{stanley-ec1} for further discussion about that example. 

The coefficients of the $\varepsilon$-Chow polynomial are \emph{sums} of entries of the flag $h$-vector, so it is possible that the presence of several negative entries may be compensated for by a few positive entries, thereby yielding a non-negative Chow polynomial. We pose the following question.

\begin{question}
    Is the $\varepsilon$-Chow polynomial of an Eulerian poset non-negative?
\end{question}

By the preceding discussion, if there is an example answering the above question in the negative, it has to be rank $7$ or above, and we expect it to have a very complicated shape.

\subsection{Unimodality and \texorpdfstring{$\gamma$}{gamma}-positivity}

Without imposing additional restrictions, the $\varepsilon$-Chow polynomial of an Eulerian poset need not be unimodal. However, as we will explain here, when the Eulerian poset comes from a nice geometric object, unimodality and $\gamma$-positivity follow from deep results in combinatorial algebraic geometry.

One can generalize face posets of polytopes in different ways. One such generalization appears as follows. A \emph{regular CW complex} is a (finite) collection $\Gamma$ of non-empty pairwise disjoint open subsets $\{\sigma_i\}_{i\in I}\subseteq \mathbb{R}^n$ (for some $n$) such that
\begin{enumerate}[(i)]
    \item Each closure $\overline{\sigma}_i$ is homeomorphic to a closed ball $\mathbb{B}^{n_i}$ (of some dimension $n_i$). Moreover, this homeomorphism restricted to the boundary $\partial \sigma_i$ yields a homeomorphism with the sphere $\mathbb{S}^{n_i-1}$.
    \item The boundary $\partial \sigma_i$ of each $\sigma_i$ is the union of some $\sigma_j$'s.
\end{enumerate}

The \emph{underlying space} of $\Gamma$, denoted $|\Gamma|$, is by definition the subspace of $\mathbb{R}^n$ obtained by the union of all the $\sigma_i$'s in $\Gamma$. The empty face and the full space $|\Gamma|$ are called improper cells. The \emph{face poset} of $\Gamma$ is the poset of all the cells $\sigma_i$ ordered by $\sigma_i\leq \sigma_j$ whenever $\overline{\sigma}_i \subseteq \overline{\sigma}_j$. We will denote this poset by $P(\Gamma)$. Notice that $P(\Gamma)$ is graded, and the rank function on all proper faces is given by $\rho(\sigma_i)= n_i+1$. 

If $|\Gamma|$ is homeomorphic to a sphere, we call $\Gamma$ a \emph{regular CW sphere}. These cell complexes are relevant in the present context because the face poset of a regular CW sphere is known to be Eulerian (see, e.g., \cite[Proposition~3.8.9]{stanley-ec1}). On the other hand, face posets of regular CW spheres are Cohen--Macaulay. A poset that is simultaneously Eulerian and Cohen--Macaulay is often called a \emph{Gorenstein* poset} (the asterisk being part of the notation). In other words, face posets of regular CW spheres are Gorenstein*.

The order complex $\Delta(P)$ of the face poset $P=P(\Gamma)$ of a regular CW sphere $\Gamma$ is often called the \emph{barycentric subdivision} of $\Gamma$. In particular, Theorem~\ref{thm:chow-eulerian-h-vector-order-complex} says that if $P$ is the face poset of a regular CW sphere $\Gamma$, the Chow polynomial is the $h$-vector of the barycentric subdivision of $\Gamma$. 

When $\Gamma$ is a polyhedral complex, the barycentric subdivision of $\Gamma$ can be seen geometrically in a straightforward way and, moreover, it is the boundary complex of a simplicial polytope. The $h$-vectors of simplicial polytopes are known to be unimodal thanks to the $g$-theorem for simplicial polytopes, proved by Stanley \cite{stanley-g-theorem} and Billera--Lee \cite{billera-lee}. In particular, the $\varepsilon$-Chow polynomial of the face poset of a polytope is unimodal.

When $P$ is a graded bounded poset of rank $r$, one can encode the flag $h$-vector of $P$ via the \emph{$\mathbf{ab}$-index}. Formally, it is defined as the polynomial $\Psi_{P}(\mathbf{a},\mathbf{b})$, in the non-commuting variables $\mathbf{a},\mathbf{b}$, given by
    \[ \Psi_P(\mathbf{a},\mathbf{b}) = \sum_{S\subseteq [r-1]} \beta_P(S)\, u_S,\]
where $u_S := e_1\cdots e_{r-1}$, and $e_i = \mathbf{a}$ if $i\not\in S$ and $e_i=\mathbf{b}$ if $i\in S$. Notice that the transformation that takes the flag $h$-vector into the flag $f$-vector of $P$ can be rewritten via the following identity:
    \[ \Psi_P(\mathbf{a}+\mathbf{b},\mathbf{b}) = \sum_{S\subseteq [r-1]} \alpha_P(S)\, u_S.\]
A fundamental property of Eulerian posets  is that their $\mathbf{ab}$-indices can be written in the form
\[ \Psi_P(\mathbf{a},\mathbf{b}) = \Phi_P(\mathbf{a}+\mathbf{b},\mathbf{ab}+\mathbf{ba}),\]
for some polynomial $\Phi_P(\mathbf{c},\mathbf{d})$ in the non-commuting variables $\mathbf{c}$ and $\mathbf{d}$. The polynomial $\Phi_P(\textbf{c},\mathbf{d})$ is called the \emph{$\mathbf{cd}$-index}. We refer to \cite[Section~3.17]{stanley-ec1} for more details, and to \cite{bayer-cd} for a thorough exposition on the $\mathbf{cd}$-index. The following is a very deep result proved by Karu in \cite{karu}.

\begin{theorem}[Karu]\label{thm:karu}
    Let $P$ be a Gorenstein* poset. The $\mathbf{cd}$-index of $P$ has non-negative coefficients.
\end{theorem}

Using this result, one can strengthen the unimodality of $\varepsilon$-Chow polynomials of face posets of polytopes in two ways. First, this phenomenon extends to all Gorenstein* posets, and second, the stronger property of $\gamma$-positivity holds.

\begin{theorem}
    Let $P$ be a Gorenstein* poset. The $\varepsilon$-Chow polynomial of $P$ is $\gamma$-positive.
\end{theorem}

\begin{proof}
    This follows from combining Theorem~\ref{thm:chow-eulerian-h-vector-order-complex}, Theorem~\ref{thm:karu}, and an observation made by Gal in \cite[p.~237]{gal}, that the $\gamma$-polynomial of the $h$-vector of $\Delta(P)$ equals $\Phi_{P}(1,2t)$.
\end{proof}

\subsection{Open questions about Eulerian Chow polynomials}

It is natural to ask for inequalities or properties beyond $\gamma$-positivity, for example real-rootedness. The following is an equivalent reformulation of a long-standing and influential open question by Brenti and Welker \cite{brenti-welker}.

\begin{question}[Brenti and Welker]\label{question:eulerian-chow-real-rooted}
    Let $P$ be the face poset of a convex polytope. Is the associated $\varepsilon$-Chow polynomial real-rooted?
\end{question}

They proved that the answer to the above question is affirmative if $P$ is the face poset of a \emph{simplicial} convex polytope. Furthermore, if $P$ is the face poset of a simplicial homology sphere (or, more generally, a Boolean cell complex), then one can apply results by Nevo, Petersen, and Tenner \cite{nevo-petersen-tenner} to characterize further conditions that the $\varepsilon$-Chow polynomial must satisfy.

Even more broadly, Athanasiadis and Kalampogia-Evangelinou have asked whether the $\varepsilon$-Chow polynomial of a Gorenstein* poset is always real-rooted (see \cite[Question~5.2]{athanasiadis-kalampogia}). They proved that many operations that preserve the Gorenstein* property also preserve the real-rootedness of the $\varepsilon$-Chow polynomial. We also refer to the work of Athanasiadis and Tzanaki \cite{athanasiadis-tzanaki} for related results.

Numerous questions about Eulerian Chow polynomials are in order. Although we provided a concrete description of the $\varepsilon$-Chow polynomial of any Eulerian poset, it is unclear what the \emph{augmented} Chow polynomials are. It is not difficult to see that the left augmented $\varepsilon$-Chow polynomial of $P$ equals the right augmented $\varepsilon$-Chow polynomial of the dual poset $P^*$. That is, unlike the case of graded posets with $\kappa=\chi$, there is no significant distinction between the left and right augmented Chow functions. 

\begin{question}
    Let $P$ be an Eulerian poset. What do the coefficients of the right (or left) augmented $\varepsilon$-Chow polynomial enumerate?
\end{question}

We observe that the same question for the $Z$-polynomial has been raised by Proudfoot in \cite{proudfoot-kls}. 

In light of the fact that the $\varepsilon$-Chow polynomial of the face poset of a polytope is the $h$-vector of a simplicial polytope, it is natural to ask whether the left (or right) augmented Chow polynomials can be realized as $h$-vectors of simplicial polytopes as well.

\begin{question}
    Let $P$ be the face poset of a polytope. Is the left (or right) augmented $\varepsilon$-Chow polynomial of $P$ the $h$-vector of some simplicial polytope?
\end{question}

Based on substantial computational evidence, we believe that the answer to the above question is likely affirmative. A related open question by Brenti \cite[Problem~2.9]{brenti-opac} asks whether any monic palindromic polynomial that has only negative real zeros is necessarily the $h$-vector of a simplicial polytope. Our experiments even suggest that the \emph{augmented} $\varepsilon$-Chow polynomials of face posets of polytopes are real-rooted.

\section{Coxeter Chow functions of Bruhat intervals}\label{sec:six}

In this section, we will study various combinatorial aspects of the Chow function arising from the $R$-polynomials in intervals of the Bruhat order of a Coxeter group $(W,S)$. 

\subsection{A short recapitulation} We briefly recall the necessary definitions in this setting. For a more detailed background, we refer to Bj\"orner and Brenti's book \cite{bjorner-brenti}.

Let $(W,S)$ be a Coxeter system and $T = \{wsw^{-1} : w\in W, s\in S\}$ be the set of its reflections (in contrast with the set $S$ of \emph{simple} reflections). The Bruhat order is the poset defined on $W$ with relations $u \leq v$ whenever there exist $w_0,\ldots, w_n$ such that $w_0 = u$, $w_n = v$, and $w_i^{-1}w_{i+1} = t \in T$ for all $i= 1,\ldots, n-1$. This poset is graded, with rank function given by $\rho(w) = \ell(w)$, where $\ell(w)$ is the length of any reduced word equal to $w$. Furthermore, Bruhat intervals are Eulerian and shellable, and hence Gorenstein* (see \cite[Proposition~2.7.5]{bjorner-brenti}).

We also define the \emph{Bruhat graph} $B(W)$ to be the directed graph whose vertices are the elements of $W$ and whose edges are of the form $u \rightarrow v$ if $u^{-1}v = t \in T$. Similarly, we define the Bruhat graph $B(u,v)$ for $u,v \in W$ by restricting to the interval $[u,v] \subset W$. Notice that the Bruhat graph has more edges than the Hasse diagram of the corresponding Bruhat order.

Given an element $w \in W$, the elements of the set $D_R(w) := \{s \in S : \rho(ws) < \rho(w)\}$ are called the \emph{right descents of $w$}. 

The \emph{$R$-polynomial} of an interval in the Bruhat poset is defined recursively as follows. For $s \in D_R(v)$,

\[
R_{uv}(x) = \begin{cases}
1 &\text{if $u=v$},\\
R_{us,vs}(x) &\text{if $s \in D_R(u)$},\\
xR_{us,vs}(x) + (x-1)R_{u,vs}(x) &\text{if $s \notin D_R(u)$}.
\end{cases}
\]
\begin{figure}\[
\begin{tikzcd}[row sep = huge,column sep = small]
                                          &                                                                                      &                                          & s_1s_2s_3s_2s_1                                                                                         &                                             &                                                                                      &                                         \\
                                          & s_2s_1s_2s_3 \arrow[rru]                                                             & s_3s_2s_1s_2 \arrow[ru]                  &                                                                                                         & s_3s_2s_1s_3 \arrow[lu]                     & s_1s_3s_2s_3 \arrow[llu]                                                             &                                         \\
s_2s_1s_2 \arrow[ru] \arrow[rru]          & s_2s_1s_3 \arrow[u] \arrow[rrru]                                                     & s_1s_2s_3 \arrow[lu] \arrow[rrru]        &                                                                                                         & s_3s_2s_1 \arrow[llu] \arrow[u]             & s_1s_3s_2 \arrow[lllu] \arrow[u]                                                     & s_3s_2s_3 \arrow[llu] \arrow[lu]        \\
s_2s_1 \arrow[u] \arrow[ru] \arrow[rrrru] &                                                                                      & s_1s_2 \arrow[llu] \arrow[u] \arrow[rru] & s_1s_3 \arrow[llu] \arrow[lu] \arrow[ru] \arrow[rru]                                                    & s_2s_3 \arrow[llu] \arrow[lllu] \arrow[rru] &                                                                                      & s_3s_2 \arrow[llu] \arrow[lu] \arrow[u] \\
                                          & s_1 \arrow[lu] \arrow[ru] \arrow[rru] \arrow[rrrruuu, dotted] \arrow[rrruuu, dotted] &                                          & s_2 \arrow[lu] \arrow[ru] \arrow[rrru] \arrow[lllu]                                                     &                                             & s_3 \arrow[ru] \arrow[lu] \arrow[llu] \arrow[lllluuu, dotted] \arrow[llluuu, dotted] &                                         \\
                                          &                                                                                      &                                          & 1 \arrow[u] \arrow[llu] \arrow[rru] \arrow[llluuu, dotted] \arrow[rrruuu, dotted] \arrow[uuuuu, dotted] &                                             &                                                                                      &                                                       
\end{tikzcd}
\]
\caption{The lower ideal generated by  $w = s_1s_2s_3s_2s_1$ in the Bruhat order of $\mathfrak{S}_4$. The dotted arrows are the directed edges that we add when we consider the corresponding Bruhat graph $B(1,w)$.}\label{bruhat interval S4}
\end{figure}
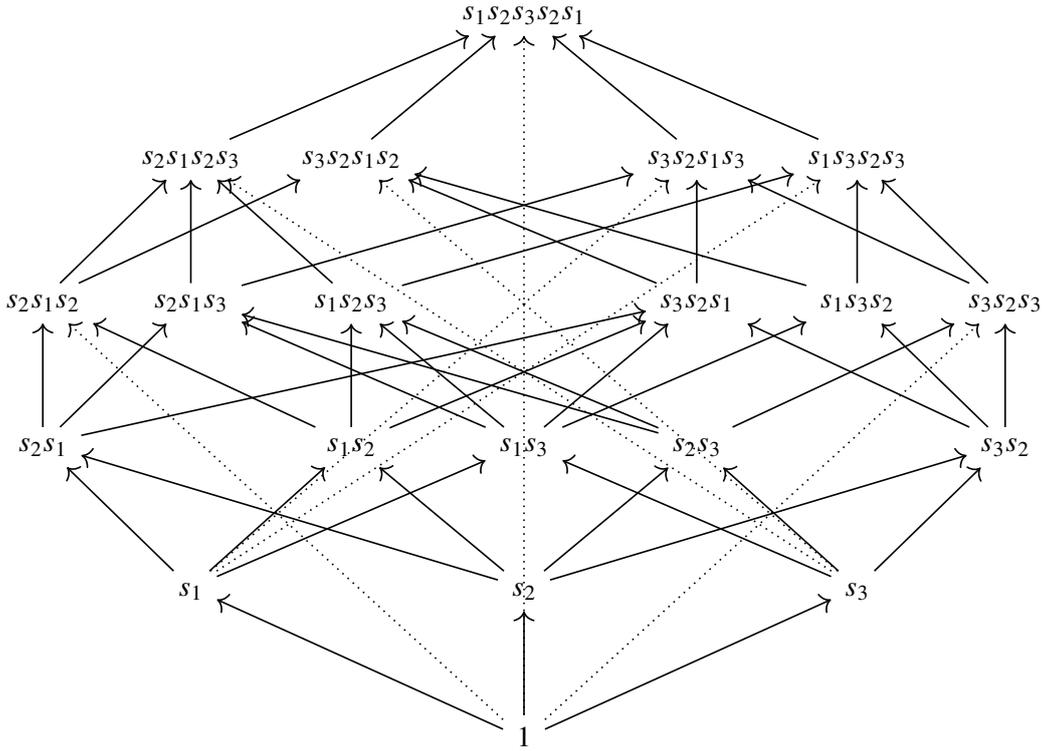
\begin{example}\label{ex: recurring bruhat example}
Consider the symmetric group $\mathfrak{S}_4$, generated by the simple reflections $S = \{s_1,s_2,s_3\}$, where $s_i$ is the transposition $(i\; i+1)$. The set of reflections is then \[T = \{s_1,s_2,s_3,s_1s_2s_3,s_2s_3s_2,s_1s_2s_3s_2s_1\}.\] For the word $w = s_1s_2s_3s_2s_1$, we draw both the Hasse diagram of the interval $[1,w]$ in the Bruhat order and the corresponding Bruhat graph in Figure \ref{bruhat interval S4}. One can recursively compute that $R_{1,w}(x) = x^5 - 3x^4 + 5x^3 - 5x^2 + 3x - 1$.
\end{example}
One can check that the $R$-function is a $P$-kernel in the incidence algebra of Bruhat orders (see for example \cite[Exercise~11]{bjorner-brenti}) and thus we may define the \emph{Coxeter Chow function}, or \emph{$R$-Chow function} for brevity, of a Coxeter group.

Historically, the Coxeter case is the original motivating example \cite{kazhdan-lusztig} that led to the development of KLS theory. A striking difference that sets this case apart from the ones studied in Sections~\ref{sec:four} and~\ref{sec:five} is that it is still not known whether the KLS functions are combinatorially invariant. We will discuss this briefly in Section~\ref{sec: combinatorial invariance} below. Moreover it is known that intervals of the Bruhat order are not necessarily Bruhat orders of smaller Coxeter groups, i.e., this property is not hereditary, whilst both being a geometric lattice or an Eulerian poset are hereditary properties.

\subsection{A formula for the Coxeter Chow function}
We now proceed to compute an explicit formula for the $R$-Chow function. We refer the interested reader to \cite[Section~5]{bjorner-brenti}. Further combinatorial formulas can be found in \cite{brenti-combinatorial1,brenti-combinatorial2}.

In the remainder of this section, $\Phi^+$ will denote the positive roots of $W$. It will be useful to work with the following classical reparameterization of the $R$-polynomials.

\begin{proposition}[{\cite[Proposition~5.3.1]{bjorner-brenti}}]\label{prop: def-R-tilde}
    Let $u,v \in W$. Then, there exists a unique polynomial $\widetilde{R}_{uv}(x) \in \mathbb{N}[x]$ such that
    \[
    R_{uv}(x) = x^{\rho_{uv}/2}\widetilde{R}_{uv}(x^{1/2}-x^{-1/2}).
    \]
\end{proposition}

Recall that a total ordering $<$ on $\Phi^+$ is a \emph{reflection ordering} if for all $\alpha,\beta \in \Phi^+$ and $\lambda,\mu >0$ such that $\lambda \alpha + \mu \beta \in \Phi^+$, then 
\[ 
\text{either $\qquad
\alpha < \lambda \alpha + \mu \beta < \beta,
\qquad$
or
$\qquad
\beta < \lambda \alpha + \mu \beta < \alpha
$}.\]

Since there exists a bijection between $\Phi^+$ and $T$, one can also describe a reflection ordering on $T$. Given a reflection ordering on $\Phi^+$ and a path $\Delta = (a_0,\ldots, a_r)$ in $B(u,v)$ of length $r$, we define the \emph{edge set} of $\Delta$ to be
\[
E(\Delta) = \{a_{i-1}^{-1}a_i \mid i=1,\ldots, r \}\subseteq T
\]
and say that $i \in \{1, \ldots, r-1\}$ is a \emph{descent} of $\Delta$ if $a_{i-1}^{-1}a_i > a_i^{-1}a_{i+1}$. We denote the set of descents of a path with $D(\Delta)$ and set $\des(\Delta) = |D(\Delta)|$. Similarly, one can define the \emph{ascent set} of $\Delta$ and the quantity $\asc(\Delta) = \ell(\Delta) - \des(\Delta) -1$.

\begin{theorem}[{\cite{dyer}, \cite[Theorem~5.3.4]{bjorner-brenti}}]\label{thm: comb-R-tilde}
    For every $u,v \in W$,
    \[
    \widetilde{R}_{uv}(x) = \sum_{\substack{\Delta \in B(u,v) \\ \des (\Delta) = 0}} x^{\ell(\Delta)}. 
    \]
\end{theorem}

Combining Proposition \ref{prop: def-R-tilde} and Theorem \ref{thm: comb-R-tilde}, we have the following immediate consequence.
\begin{theorem}
Let $W$ be a Coxeter group with a reflection order $<$ and two elements $u, v \in W$. Then, 
\[
R_{uv}(x) = \sum_{\substack{\Delta \in B(u,v) \\ \des (\Delta) = 0}} x^{\frac{\rho_{uv}-\ell(\Delta)}{2}}(x-1)^{\ell(\Delta)}.
\]
\end{theorem}

Now we have all the ingredients to state and prove the combinatorial description of the Coxeter Chow function.

\begin{theorem}\label{thm: comb-coxeter-H}
Let $W$ be a Coxeter group with a reflection order $<$ and two elements $u, v \in W$. Then, 
\[
\H_{uv}(x) = \sum_{\Delta \in B(u,v)}x^{\frac{\rho_{uv}-\ell(\Delta)}{2} + \asc(\Delta)} = \sum_{\Delta \in B(u,v)}x^{\frac{\rho_{uv}-\ell(\Delta)}{2} + \des(\Delta)}.
\]
\end{theorem}

\begin{proof}
The formulas are clearly true if $\rho_{uv} \leq 1$. For $\rho_{uv} \geq 2$, we start by unravelling the recursion and writing $\H$ only in terms of $\overline{R}$ (notice the $-1$ in the exponent of $(x-1)$ because we reduced it).
\begin{align*}
\H_{uv}(x) &= \sum_{\substack{\mathcal{U} = \{u_1,\ldots , u_r\} \\u = u_0 < \cdots < u_{r+1} = v}} \prod_{i=1}^{r+1}\sum_{\substack{\Delta_i \in B(u_{i-1},u_i) \\ \des(\Delta_i) = 0}}x^{\frac{\rho_{u_{i-1}u_i} - \ell(\Delta_i)}{2}}(x-1)^{\ell(\Delta_i) - 1}\\
&= \sum_{\mathcal{U}}\sum_{\substack{\Delta \in B(u,v) \\ \mathcal{U} \subseteq \Delta \\ \des(\Delta) \subseteq \mathcal{U}}}x^{\frac{\rho_{uv} - \ell(\Delta)}{2}}(x-1)^{\ell(\Delta) - (r+1)} \\
&= \sum_{\Delta \in B(u,v)} x^{\frac{\rho_{uv} - \ell(\Delta)}{2}}\sum_{\substack{\mathcal{U} \subseteq \Delta\setminus \{u,v\} \\ \mathcal{U} \supseteq \des(\Delta)}} (x-1)^{\ell(\Delta) - (r+1)}.
\end{align*}
The inner sum is equal to $x^{\asc(\Delta)}$ as 
\begin{align*}
x^{\asc(\Delta)} = x^{\ell(\Delta) - \des(\Delta) - 1} &= (x-1+1)^{\ell(\Delta) - \des(\Delta) - 1}\\
&= \sum_{j = 0}^{\ell(\Delta) - \des(\Delta) - 1}\binom{\ell(\Delta) - \des(\Delta) - 1}{j} (x-1)^{\ell(\Delta) - \des(\Delta) - 1 - j} \\
&= \sum_{r = \des(\Delta)}^{\ell(\Delta) - 1}\binom{\ell(\Delta) - \des(\Delta) - 1}{r - \des(\Delta)} (x-1)^{\ell(\Delta) - 1 - r}.
\end{align*}
This is clearly counting the subsets $\mathcal{U}$ as desired. Then, 
\[
\H_{uv}(x) = \sum_{\Delta \in B(u,v)}x^{\frac{\rho_{uv} - \ell(\Delta)}{2} + \asc(\Delta)}.
\]
The statement with descents comes from the fact that $\des(\Delta) = \ell(\Delta) - \asc(\Delta) - 1$ and that the polynomial $\H_{uv}(x)$ is symmetric with center of symmetry $\frac{\rho_{uv}-1}{2}$.
\end{proof}

\begin{example}
Consider the interval from Example \ref{ex: recurring bruhat example}. We have $73$ paths in the directed graph $B(1,w)$. Let us also fix the reflection order 
\[
s_1<s_1s_2s_1<s_1s_2s_3s_2s_1<s_2<s_2s_3s_2<s_3.
\]
A straightforward (and not enlightening) computation shows that there are
\begin{itemize}
    \item 1 path of length 1 with 0 descents,
    \item 2 paths of length 3 with 0 descents,
    \item 4 paths of length 3 with 1 descent,
    \item 2 paths of length 3 with 2 descents,
    \item 1 path of length 5 with 0 descents,
    \item 14 paths of length 5 with 1 descent,
    \item 34 paths of length 5 with 2 descents,
    \item 14 paths of length 5 with 3 descents,
    \item 1 path of length 5 with 4 descents.
\end{itemize}
By Theorem \ref{thm: comb-coxeter-H}, we can conclude that
\begin{align*}
\H_{1,w}(x) &= x^2 \\
&\,\,\,\,\,\,\,+ 2x^1 + 4x^2 + 2x^3\\ 
&\,\,\,\,\,\,\,+ x^0 + 14x^1 + 34x^2 + 14x^3 + x^4\\
&=1 + 16x + 39x^2 + 16x^2 + x^4.
\end{align*}
\end{example}

\begin{remark}
    By considering the whole Bruhat poset on $\mathfrak{S}_n$, and denoting the corresponding $R$-Chow polynomial by $\H_{\mathfrak{S}_n}(x)$, we obtain the following first few values:
    {\footnotesize\begin{align*}
         \H_{\mathfrak{S}_n}(x) = \begin{cases}
            1 & n = 1,\\
            1 & n = 2,\\
            x^2 + 3\,x + 1 & n=3,\\
            x^{5} + 20 \, x^{4} + 84 \, x^{3} + 84 \, x^{2} + 20 \, x + 1 & n = 4,\\
            x^{9} + 115 \, x^{8} + 2856 \, x^{7} +  21429 \, x^{6} + 56840 \, x^{5} + 56840 \, x^{4} + 21429 \, x^{3} + 2856 \, x^{2} + 115 \, x + 1 & n = 5.
        \end{cases}
    \end{align*}}

The sequences of coefficients of these polynomials do not appear in the OEIS \cite{oeis}. In the authors' opinion, providing a closed formula for these polynomials or, at least, an efficient way of computing them would be very interesting. 
\end{remark}

\begin{remark}
    We have not been able to find a nice analogue of Theorem~\ref{thm: comb-coxeter-H} for the right and left augmented Chow functions arising in this setting. 
\end{remark}

\subsection{The complete $\mathbf{cd}$-index and gamma-positivity}

The groundbreaking work of Elias and Williamson~\cite{elias-williamson}, who proved the non-negativity conjecture for Kazhdan--Lusztig polynomials of Bruhat intervals of Coxeter groups, implies via Theorem~\ref{thm:kls-positive-chow-unimodal} that Coxeter Chow polynomials are unimodal. It is reasonable to inquire whether stronger inequalities among the coefficients hold.

As mentioned earlier, Bruhat intervals are Eulerian posets and, moreover, they admit a special shelling \cite[Theorem~2.7.5]{bjorner-brenti}. In particular, they are Gorenstein*, so that one can apply Karu's result in Theorem~\ref{thm:karu} to conclude that their $\mathbf{cd}$-index has non-negative coefficients. We refer to Reading's \cite{reading} article for a thorough study of the $\mathbf{cd}$-index of Bruhat intervals. Despite its relevance in this context, the $\mathbf{cd}$-index is not enough to compute Kazhdan--Lusztig or Coxeter Chow polynomials. In order to do this, one needs to define a more complicated counterpart of the $\mathbf{cd}$-index called the \emph{complete $\mathbf{cd}$-index}. This was introduced in the work of Billera and Brenti \cite{billera-brenti}. Following their notation, the complete $\mathbf{cd}$-index of the interval $[u,v]$ in the Bruhat order of the Coxeter group $W$ is denoted by $\widetilde{\psi}_{uv}$. This is a polynomial in the non-commuting variables $\mathbf{c}$ and $\mathbf{d}$. We will not require the technical subtleties behind the actual definition of the complete $\mathbf{cd}$-index, and we refer the reader to Billera and Brenti's paper for that purpose. However, we do state one of their theorems, which we will use to describe the Coxeter Chow function as a specialization of the complete $\mathbf{cd}$-index.

\begin{theorem}[{\cite[Proposition~2.9]{billera-brenti}}]\label{psiab}
    Let $u,v \in W$ and $u<v$. Then,
    \[
    \widetilde{\psi}_{uv}(\mathbf{a}+\mathbf{b},\mathbf{a}\mathbf{b}+\mathbf{b}\mathbf{a}) = \sum_{\Delta \in B(u,v)}w(\Delta),
    \]
    where each $w(\Delta)$ is a non-commutative monomial of degree $\ell(\Delta)-1$ in the variables $\mathbf{a}$ and $\mathbf{b}$, having $\mathbf{a}$ as the $i$-th letter from the left if $i\notin D(\Delta)$, and $\mathbf{b}$ if $i \in D(\Delta)$. 
\end{theorem}

The preceding result allows us to obtain $\H_{uv}(x)$ as a specialization of $\widetilde{\psi}_{uv}(\mathbf{c},\mathbf{d})$.

\begin{theorem}\label{thm:chow-from-cd}
    Let $u,v \in W$ and $u<v$. Then,
    \[
    \H_{uv}(x) = x^{\frac{\rho_{uv}-1}{2}}\,\widetilde{\psi}_{uv}\left(x^{-1/2}+x^{1/2},2 \right).
    \]
\end{theorem}

\begin{proof}
    Using the formula in Theorem \ref{psiab}, we set $\mathbf{a} = x^{-\frac{1}{2}}$ and $\mathbf{b} = x^{\frac{1}{2}}$. Then, for a given path in the Bruhat graph $\Delta$ we get
    \[
    w(\Delta)\big{|}_{\mathbf{a} = x^{-1/2},\mathbf{b} = x^{1/2}} = x^{-\frac{1}{2}(\ell(\Delta) - \des(\Delta)-1) + \frac{1}{2}\des(\Delta)}.
    \]
    The result then follows directly from our Theorem~\ref{thm: comb-coxeter-H}.
\end{proof}

A remarkable consequence of Theorem~\ref{thm:chow-from-cd} is that the $\gamma$-polynomial associated to $\H_{uv}(x)$ is a non-negative specialization of the complete $\mathbf{cd}$-index $\widetilde{\psi}$.

\begin{corollary}\label{coro:complete-cd-chow-gamma}
    Let $u,v\in W$ and $u < v$. Then, the $\gamma$-polynomial associated to $\H_{uv}(x)$ can be obtained from the complete $\mathbf{cd}$-index as
    \[ \gamma(\H_{uv},x^2) = x^{\rho_{uv}} \widetilde{\psi}_{uv}(x^{-1}, 2).\]
\end{corollary}

\begin{proof}
    By definition, we have that
    \[\H_{uv}(x) = (1+x)^{\rho_{uv}-1} \,\gamma\left(\H_{uv}, \frac{x}{(x+1)^2}\right).\]
    In particular, using the change of variable $y^2 = \frac{x}{(1+x)^2}$ or, equivalently, $y= \left(x^{1/2}+x^{-1/2}\right)^{-1}$, we have the following chain of equalities:
    \[
        \gamma(\H_{uv}, y^2) = \frac{1}{(1+x)^{\rho_{uv}-1}} \H_{uv}(x)\\
        =  \frac{1}{(1+x)^{\rho_{uv}-1}} x^{\frac{\rho_{uv}-1}{2}}\,\widetilde{\psi}_{uv}(y^{-1},2)= y^{\rho_{uv}}\, \widetilde{\psi}_{uv}(y^{-1},2),
    \]
    as desired.
\end{proof}

\begin{example}
    We continue with Example \ref{ex: recurring bruhat example}. The computations in \cite[Example~2.4]{billera-brenti} show that
    \[
    \widetilde{\psi}_{1,w}(\mathbf{c},\mathbf{d}) = \mathbf{c}^4 + \mathbf{d}\mathbf{c}^2 + 2\mathbf{c}\mathbf{d}\mathbf{c} + 2\mathbf{c}^2\mathbf{d} + 2\mathbf{d}^2 + 2\mathbf{c}^2 + \mathbf{1}.\]
    Then,
    \begin{align*}
    \H_{1,w}(x) &= x^2 \left[\frac{(x+1)^4}{x^2} + 2\frac{(x+1)^2}{x} + 4\frac{(x+1)^2}{x} + 4\frac{(x+1)^2}{x} + 8 + 2\frac{(x+1)^2}{x} + 1 \right]\\
    &= (x+1)^4 + 12x(x+1)^2 + 9x^2\\
    &= x^4 + 16x^3 + 39x^2 + 16x + 1.
    \end{align*}
    Indeed, $\gamma(\H_{uv},x) = 1 + 12x + 9x^2$.
\end{example}

Billera and Brenti conjecture in \cite[Conjecture~6.1]{billera-brenti} that the coefficients of the complete $\mathbf{cd}$-index of any Bruhat interval are non-negative. The preceding result implies that if their conjecture is true, then Coxeter Chow polynomials are $\gamma$-positive. In other words, their conjecture implies the following conjecture.

\begin{conjecture}
    Coxeter Chow polynomials of Bruhat intervals of Coxeter groups are always $\gamma$-positive.
\end{conjecture}

As Billera and Brenti note in \cite[Section~6]{billera-brenti}, the result by Karu in \cite{karu} implies that \emph{some} coefficients of the complete $\mathbf{cd}$-index are non-negative. Furthermore, Karu proved in another paper \cite{karu-complete-cd} the non-negativity of further coefficients of the complete $\mathbf{cd}$-index. However, it is unclear whether the currently known inequalities concerning coefficients of the complete $\mathbf{cd}$-index are enough to prove our $\gamma$-positivity conjecture. Moreover, in striking similarity to Conjecture~\ref{conj:char-chow-real-rooted} and Question~\ref{question:eulerian-chow-real-rooted}, we dare to formulate the following (much more ambitious) conjecture.

\begin{conjecture}
    Coxeter Chow polynomials of Bruhat intervals of Coxeter groups are always real-rooted.
\end{conjecture}

We have verified the validity of this conjecture on various small cases, including all intervals of symmetric groups $\mathfrak{S}_n$ for $n\leq 6$.

\subsection{Combinatorial invariance}\label{sec: combinatorial invariance}
The combinatorial invariance of Kazhdan--Lusztig polynomials is a long-standing conjecture attributed independently to Lusztig and Dyer \cite{dyer1987hecke}.

\begin{conjecture}[{Combinatorial invariance conjecture \cite{dyer1987hecke}}]\label{conj:combinv}
The Kazhdan--Lusztig polynomials of Coxeter groups are combinatorially invariant.
\end{conjecture}

The conjecture has attracted much research recently and has been resolved in a number of cases, see for example \cite{dyer,brenti-combinatorial1,B04,BCM06,I06,I07,Bre09,BMS16,Mar16,Mar18,P21,Davies2021-nx,BBDVW22,BLP23,BG23,BG24}.  We would like to point out that Conjecture~\ref{conj:combinv} can be recast into the theory of Chow functions.

\begin{theorem}\label{thm:comb-invariance}
    The combinatorial invariance conjecture for Kazhdan--Lusztig polynomials of Coxeter groups is equivalent to the combinatorial invariance conjecture for Coxeter Chow functions.
\end{theorem}

\begin{proof}
    By definition, the Kazhdan--Lusztig polynomials of Coxeter groups determine and are determined by the $R$-polynomials, and the $R$-polynomials determine and are determined by the $R$-Chow functions.
\end{proof}

It would be interesting if the combinatorial invariance conjecture for $R$-Chow functions can shed some light on Conjecture~\ref{conj:combinv}.

%\section*{Acknowledgements}

%The authors wish to thank Christos Athanasiadis, Louis Billera, Tom Braden, Petter Br\"and\'en, Chris Eur, Eva Feichtner, June Huh, Matt Larson, Mateusz Michalek, Nicholas Proudfoot, Matthew Stevens, and Botong Wang for insightful conversations. Part of these results were obtained while the authors were visiting the Oberwolfach Research Institute of Mathematics MFO, as part of the workshop ``Arrangements, matroids, and logarithmic vector fields''. We are grateful to MFO and the organizers of the workshop.

\bibliographystyle{amsalpha0}
\bibliography{bibliography}

\end{document}